\documentclass[a4paper,10pt]{amsart}

\usepackage{mathtools,amssymb,amsfonts,stmaryrd,amsmath,amsthm,enumerate,hyperref,color}
\usepackage[utf8]{inputenc}
\usepackage{enumitem}

\usepackage[all,cmtip]{xy}     
\usepackage{wrapfig,tikz,tikz-cd}
\usetikzlibrary{arrows,decorations.pathreplacing}
\usepackage{caption}

\usepackage[capitalise]{cleveref}

\usepackage[parfill]{parskip}
\usepackage[margin=1.2in]{geometry}
\AtBeginDocument{\addtocontents{toc}{\protect\setlength{\parskip}{0pt}}}

\newtheorem{theorem}{Theorem}[section]
\newtheorem{lemma}[theorem]{Lemma}
\newtheorem{proposition}[theorem]{Proposition}
\newtheorem{corollary}[theorem]{Corollary}

\newtheorem{theorema}{Theorem}

\theoremstyle{definition}
\newtheorem{definition}[theorem]{Definition}
\newtheorem{remark}[theorem]{Remark}
\newtheorem{example}[theorem]{Example}

\theoremstyle{remark}

\DeclareMathOperator{\gr}{gr}
\DeclareMathOperator{\lin}{lin}
\newcommand{\fm}{\mathfrak{m}}
\newcommand{\pd}{\mathrm{projdim}}
\DeclareMathOperator{\Ext}{Ext}
\DeclareMathOperator{\Tor}{Tor}

\begin{document}

\title[Homotopy types associated to almost linear resolutions]{Homotopy types of moment-angle complexes\\ associated to almost linear resolutions} 

\date{}

\author[S.~Amelotte]{Steven Amelotte}
\address{School of Mathematics and Statistics, 4302 Herzberg Laboratories, Carleton University, Ottawa, ON, K1S 5B6, Canada}
\email{steven.amelotte@carleton.ca} 

\author[B.~Briggs]{Benjamin Briggs}
\address{Department of Mathematics, Huxley Building, South Kensington Campus, Imperial College London, 
London, SW7 2AZ, U.K.}
\email{b.briggs@imperial.ac.uk}



\begin{abstract}
We show that the Hurewicz image in the homology of a moment-angle complex, when passed through an isomorphism with the Ext-module of the corresponding Stanley--Reisner ideal, contains the linear strand of this ideal. 
This recovers and refines results of various authors identifying the homotopy type of a moment-angle complex as a wedge of spheres when the underlying ideal satisfies certain linearity properties. Going further, we study the homotopy types of moment-angle manifolds associated to Gorenstein Stanley--Reisner ideals with (componentwise) almost linear resolutions. The simplicial complexes that give rise to these manifolds are part of an even larger class that we introduce, which generalises the homological behaviour of cyclic polytopes, stacked polytopes and odd-dimensional neighbourly sphere triangulations. For these simplicial complexes the associated moment-angle manifolds are shown to be formal, having the rational homotopy type of connected sums of sphere products, and the (integral) loop space homotopy type of products of spheres and loop spaces of spheres. Along the way we establish a number of purely algebraic results, in particular generalising a result of R\"omer characterising Koszul modules so that it can be applied to modules with almost linear resolutions.
\end{abstract}

\maketitle

\setcounter{tocdepth}{1}
\tableofcontents

\section{Introduction}

The {moment-angle complex} $\mathcal{Z}_K$ associated to a simplicial complex $K$ is a fundamental object of study in toric topology. The equivariant topology of $\mathcal{Z}_K$ is intimately tied to the homological algebra of the corresponding Stanley--Reisner ring $k[K]$. In particular, writing $k[K]=S/I_K$ where $S=k[v_1,\ldots,v_m]$ is a polynomial ring, there is a natural isomorphism
\[
H^*(\mathcal{Z}_K;k)\cong \Tor^S(k[K],k).
\]
This relationship between the Betti numbers of the space $\mathcal{Z}_K$ and those of the algebra $k[K]$ was tightened in the authors' previous work~\cite{AB}, where it was shown that the minimal free resolution of $k[K]$ over $S$ can be recovered from the data of $H^*(\mathcal{Z}_K;k)$ together with a family of higher cohomology operations induced by the standard torus action on $\mathcal{Z}_K$ (see \cref{thm_min_res} below). In this paper, we consider the problem of determining the homotopy type of $\mathcal{Z}_K$ from properties of the minimal free resolution of $k[K]$.

Although the cohomology of $\mathcal{Z}_K$ is well understood, relatively little is known about its homotopy type in general. Much work in this direction has focused on determining when $\mathcal{Z}_K$ is homotopy equivalent to a wedge of spheres, which is closely related to the Golod property for $k[K]$. Beginning with work of Grbi\'c and Theriault~\cite{GT04} on the complements of coordinate subspace arrangements, this has been shown to hold, in increasing level of generality, when the Alexander dual $K^\vee$ of the simplicial complex $K$ is shifted~\cite{GT07,GT13,IK13}, vertex decomposable~\cite{GW}, shellable or sequentially Cohen--Macaulay 
\cite{IK19}, and when $K$ is an HMF complex~\cite{BG}.

Our first main result starts from the observation that all of these conditions on $K$ can be understood in terms of the minimal free resolution of the Stanley--Reisner ideal $I_K$. For example, by a classical theorem of Eagon and Reiner~\cite{ER}, $K^\vee$ is Cohen--Macaulay over $k$ if and only if $I_K$ has a linear resolution over~$S$. Generalising this to non-pure simplicial complexes and non-equigenerated ideals, Herzog and Hibi~\cite{HH} proved that $K^\vee$ is sequentially Cohen--Macaulay over $k$ if and only if $I_K$ is componentwise linear over~$S$, which in turn
is equivalent to $I_K$ being a Koszul $S$-module by a result of R\"omer~\cite{R}. More generally still, $I_K$ is a quasi-Koszul $S$-module if and only if $K$ is an HMF complex over $k$ (see \cref{prop_HMF}). Here quasi-Koszul means that $\Ext_S(I_K,k)$ is generated in degree zero as a $\Lambda=\Ext_S(k,k)$-module.

Based on this, we define the \emph{quasi-linear strand} of $I_K$ to be the $\Lambda$-submodule of $\Ext_S(I_K,k)$ generated by degree zero classes (for equigenerated ideals this coincides with the usual linear strand). Through the identification  $\widetilde{H}_*(\mathcal{Z}_K;k) \cong \Ext_S(I_K,k)$, dual to the isomorphism above, we may view the quasi-linear strand as a subspace of $H_*(\mathcal{Z}_K;k)$ and attempt to understand it topologically. Since a simply connected space has the homotopy type of a wedge of spheres precisely when its Hurewicz map is surjective, all of the above results on the homotopy type of $\mathcal{Z}_K$ are recovered by the following theorem.

\begin{theorema}[\cref{thm_Hurewicz_strand}, \cref{cor_HMF}, \cref{cor_compwise}] \label{thm_introA}
Let $K$ be a simplicial complex and let $k$ be a field of characteristic~$p\geqslant 0$. Then the quasi-linear strand of $k[K]$ is contained in (the $k$-linear span of) the image of the Hurewicz map 
\[ 
h\colon \pi_*(\mathcal{Z}_K)\longrightarrow {H}_*(\mathcal{Z}_K;k).
\]
In particular, $\mathcal{Z}_K$ is $p$-locally homotopy equivalent to a wedge of spheres whenever $I_K$ has a componentwise linear resolution, or more generally, is a quasi-Koszul $S$-module.
\end{theorema}

Above, a ``$0$-local" homotopy equivalence means a rational homotopy equivalence. It follows that $\mathcal{Z}_K$ splits as a wedge of spheres without any localisation when the hypotheses are satisfied over all fields $k=\mathbb{F}_p$, or over $\mathbb{Z}$ (see \cref{sec_hurewicz}).

Combined with Fr\"oberg's classification of quadratic Stanley--Reisner ideals with linear resolutions, \cref{thm_introA} also recovers the result from~\cite{GPTW} that $\mathcal{Z}_K$ is homotopy equivalent to a wedge of spheres when $K$ is a flag complex with chordal $1$-skeleton. Using a generalisation of Fr\"oberg's theorem due to Eisenbud, Green, Hulek and Popescu~\cite{EGHP}, we refine this result in \cref{thm_flag} to a combinatorial classification of flag complexes $K$ for which a given skeleton of $\mathcal{Z}_K$ has the homotopy type of a wedge of spheres.

We are especially interested in the homotopy types of moment-angle manifolds. By a result of Cai~\cite[Corollary~2.10]{C}, combined with a well-known theorem of Stanley~\cite[Theorem~II.5.1]{Stan}, $\mathcal{Z}_K$ is a manifold exactly when $K$ is Gorenstein over all fields $k$. In this case, $I_K$ cannot have a linear resolution and $\mathcal{Z}_K$ cannot have the homotopy type of a wedge of spheres, except in trivial situations where $\mathcal{Z}_K\simeq S^{2n+1}$ (see \cref{sec_manifolds}). It turns out, however, that in the Gorenstein setting $I_K$ often has an \emph{almost linear resolution}, meaning that the nonzero entries in all but the last differential of its minimal free resolution are linear forms. We classify these ideals in \cref{prop_alin}, which shows in particular that if $K$ is a sphere triangulation, then $I_K$ has an almost linear resolution if and only if $K$ is odd-dimensional and neighbourly. 

A typical example is the Stanley--Reisner ideal of $K=C_m$, the cycle graph or boundary of a polygon (see \cref{ex_cyclic}). The corresponding moment-angle manifolds $\mathcal{Z}_{C_m}$ are one of only a few families of examples whose homotopy types have been determined explicitly: Bosio and Meersseman~\cite{BM} showed that $\mathcal{Z}_K$ is diffeomorphic to a connected sum of products of spheres when $K=C_m$, or more generally, when $K$ is the boundary complex of a stacked polytope. Another important example comes from work of Gitler and L\'opez de Medrano~\cite{GL}, who showed that $\mathcal{Z}_K$ is a connected sum of sphere products if $K=\partial P$ for an even-dimensional neighbourly simplicial polytope $P$. The geometric results of \cite{BM} and \cite{GL} apply only to certain polytopal sphere triangulations, for which $\mathcal{Z}_K$ can be identified with a smooth manifold defined by an intersection of quadrics. It is known, however, that the vast majority of sphere triangulations are not polytopal.\footnote{For example, according to~\cite{K}, the number of simplicial spheres on 1,000,000 vertices is at least $2^{2^{694,200}}$, while less than $2^{2^{62}}$ of them are polytopal! See \cref{ex_nbrly} for asymptotic bounds on the number of simplicial and polytopal $n$-spheres on $m$ vertices.} On the other hand, neighbourly sphere triangulations abound, and it is conjectured that \emph{most} triangulations of any odd-dimensional sphere are neighbourly when the number of vertices is sufficiently large (see \cref{ex_nbrly}).

Just as ideals in $S$ with (componentwise) linear resolutions are generalised by those which are quasi-Koszul $S$-modules, ideals with (componentwise) almost linear resolutions belong to a much larger class which we call \emph{almost quasi-Koszul} modules; see \cref{sec_manifolds}. The collection of Gorenstein complexes $K$ for which $I_K$ is almost quasi-Koszul includes all odd-dimensional neighbourly sphere triangulations, the boundary complexes of all cyclic and stacked polytopes, and many Gorenstein complexes for which $|K|$ is non-simply connected. Moreover, we show that this collection is closed under connected sums of simplicial complexes and stellar subdivision of facets (see \cref{sec_examples}). In particular, unlike neighbourly sphere triangulations, these $K$ may have missing faces of arbitrary size. The next theorem describes the homotopy types of the corresponding moment-angle manifolds up to a twist in the attaching map of the top cell.

\begin{theorema}[\cref{thm_conn_sum}] \label{thm_introB}
Let $K\neq\partial\Delta^{m-1}$ be an $(n-1)$-dimensional Gorenstein$^*$ complex on~$m$ vertices. If $I_K$ is an almost quasi-Koszul module, then the following hold:
\begin{enumerate}[label={\normalfont(\arabic*)}]
\item $\mathcal{Z}_K$ is an $(m+n)$-manifold with $(m+n-1)$-skeleton homotopy equivalent to a wedge of spheres;
\item $\mathcal{Z}_K$ is rationally homotopy equivalent to a connected sum of sphere products with two spheres in each product;
\item $\Omega\mathcal{Z}_K$ is homotopy equivalent to a product of spheres and loop spaces of spheres.
\end{enumerate}
\end{theorema}

Applications to formality, LS category, minimal non-Golodness, Pontryagin rings and Yoneda algebras are deduced in \cref{sec_manifolds}. Given that most triangulations of an odd-dimensional sphere are conjectured to be neighbourly, \cref{thm_introB} suggests that formality may be typical or generic for moment-angle manifolds. This should be contrasted with results of Denham and Suciu in~\cite[\S8.3]{DS}, where it is speculated that most moment-angle complexes are not formal.

Many of our arguments are based on homological properties of the Stanley--Reisner ring $k[K]$, and in \cref{se_CAL_algebra} we draw out these properties and prove a number of purely algebraic results. The prototype is a result of R\"omer \cite{R} that characterises Koszul modules over Koszul rings as those modules with componentwise linear resolutions; we establish an analogue characterising modules with resolutions that are componentwise linear up to a given degree, in terms of partial exactness of the linear part (see \cref{thm_complin_steps}). This result can in particular be applied to Gorenstein ideals in polynomial rings, which are essentially never componentwise linear but may be componentwise almost linear.

\subsection*{Acknowledgements} 
We gratefully acknowledge our use of the \emph{Manifold Page} created by Frank Lutz \cite{ManifoldPage}, which contains a huge library of manifold triangulations to which our results can be applied; to analyze these examples we also made extensive use of {\tt Macaulay2}~\cite{GS}. We also thank Fedor Vylegzhanin for drawing our attention to the work~\cite{CS}. The first author was partially supported by the Fields Institute for Research in Mathematical Sciences. The second author was partly funded by the European Union under the Grant Agreement no.\ 101064551, and partly by the National Science Foundation under Grant No.\ DMS-1928930 while hosted at SLMath (formerly MSRI).

\subsection*{Notations and conventions:}  
A simplicial complex $K$ on the vertex set $[m]=\{1,\ldots,m\}$ is a nonempty collection of subsets of $[m]$, closed under passing to subsets. Unless otherwise stated, we will assume that $K$ has no ghost vertices; that is, $\{i\}\in K$ for all $i\in[m]$. For $U\subseteq [m]$, $K_U$ denotes the \emph{full subcomplex} $\{\sigma\in K \,:\, \sigma\subseteq U\}$.
The \emph{Alexander dual} of $K$ is the simplicial complex $K^\vee=\big\{ [m]\smallsetminus F \, :\, F\notin K\big\}$. For each integer $q$, we write $K^\vee(q)=\{I\, :\, I\subseteq J\in K^\vee\text{ with }|J|=q\}$ for the (pure) subcomplex of $K^\vee$ generated by faces of size~$q$.

Throughout, $k$ is a field of any characteristic (occasionally, when pointed out, we may take $k$ to be the integers). Having fixed $m$, we write $S$ for the polynomial ring $k[v_1,\ldots,v_m]$ over $k$. The \emph{Stanley--Reisner ring} (over $k$)  of a simplicial complex $K$ on the vertex set $[m]$ is the quotient
\[
k[K]=S/I_K,\quad \text{with}\quad I_K=\big(v_{i_1}\!\cdots v_{i_q} \,:\, \{i_1,\ldots,i_q\} \notin K\big).
\]

\section{Free resolutions and linearity}
\label{se_lin_background}
Let $S$ be a standard graded $k$-algebra with maximal graded ideal $\fm$ (we are interested mostly in the polynomial case, but it will cost us little to be more general). In this section we briefly review various linearity properties for graded $S$-modules. 

In later sections, we will consider how linearity properties of Stanley--Reisner ideals are reflected in the topology of their corresponding moment-angle complexes. In \cref{se_CAL_algebra}, we further will develop the theory of componentwise almost linear ideals by building on some ideas of Iyengar and R\"omer.

\subsection*{Componentwise linear resolutions} Let $M$ be a finitely generated graded $S$-module. The \emph{graded Betti numbers} 
\[ \beta_{i,j}(M)=\dim_k\Tor^S_{i}(M,k)_j \]
describe the ranks of the free $S$-modules appearing in  the minimal graded free resolution $F\to M$:\\ 
\[
\cdots\to F_{i} \to\cdots\to F_1 \to F_0 \to M \to 0, \qquad F_{i} = \bigoplus_{j\in\mathbb{Z}}S(-j)^{\beta_{i,j}(M)}.
\]
We say that $M$ has a \emph{linear resolution}
if there exists an integer $d$ such that each $F_i$ is generated in degree $i+d$, or in other words,
\[ \beta_{i,j}(M)=0 \; \text{ for } j\neq i+d. \]
In this case, $M$ is generated in degree $d$ and the differentials in the minimal free resolution can be represented by matrices of linear forms.

More generally, $M$ is said to satisfy condition $N_{d,p}$ if it is generated in degree $d$ and its resolution $F$ is linear in the first $p$ steps, meaning that
\[ 
\beta_{i,j}(M)=0 \; \text{ for } j\neq i+d \text{ when } i<p. 
\]
For example, $M$ satisfies condition $N_{d,1}$ if it is generated in degree $d$, and $M$ satisfies condition $N_{d,2}$ if, additionally, it is linearly presented. Note that $M$ has a linear resolution if $M$ satisfies condition $N_{d,p}$ for all $p$. We say that $M$ has an \emph{almost linear resolution} if $M$ satisfies condition $N_{d,p}$ for $p=\pd_S(M)$, the projective dimension of $M$. 

The \emph{Green--Lazarsfeld index} of $M$, denoted $\operatorname{index}(M)$, is an invariant that measures how far along the minimal free resolution a first nonlinear syzygy occurs. Explicitly, for a graded $S$-module $M$ generated in degree $d$,
\[ \operatorname{index}(M)= \sup\{p\in\mathbb{Z} \,:\, M \text{ satisfies condition } N_{d,p} \}. \]
Thus, $M$ has a linear resolution if and only if $\operatorname{index}(M)=\infty$, and $M$ has an almost linear resolution if and only if $\operatorname{index}(M) \geqslant \pd_S(M)$.
 
 Going beyond the equigenerated case, Herzog and Hibi introduced the notion of componentwise linearity in~\cite{HH}. We write $M_{\langle d\rangle} = SM_d$ for the $S$-submodule generated by the degree $d$ part of $M$, and likewise $M_{\langle \leqslant d\rangle}= SM_{ \leqslant d}$. Using this notation, we say that $M$ is \emph{componentwise linear} if $M_{\langle d\rangle}$ has a linear resolution for all $d$. Similarly, $M$ is \emph{componentwise linear in the first $p$ steps}, or \emph{componentwise $N_{d,p}$}, if the module $M_{\langle d\rangle}$ satisfies condition $N_{d,p}$ for all $d$.

\subsection*{The linear part}  A related linearity condition is defined in terms of the linear part of a complex, introduced in~\cite{EFS}.
Let $F$ be a minimal complex of graded $S$-module. We first set $F_{\langle \leqslant d\rangle}$ to be the free subcomplex of $F$ with $(F_{\langle \leqslant d\rangle})_n= (F_{n})_{ \langle \leqslant d+n\rangle}$.
With this we define the $d$-linear strand to be the subquotient $F_{\langle d\rangle} = F_{\langle \leqslant d\rangle}/F_{\langle \leqslant d-1\rangle}$. The \emph{linear part} of $F$ is then the direct sum of its linear strands
\[
\lin(F) = \bigoplus_d F_{\langle d\rangle}.
\]
Intuitively, $\lin(F)$ is simply the complex with  $\lin_n(F)=F_n$ and with differential $\partial^{\lin}$ obtained by deleting all nonlinear entries from homogeneous matrices representing the differential of $F$.

Alternatively, we may define $\lin(F)$ as the associated graded complex with respect to its $\fm$-adic filtration, and with differential inherited from the factorisation $\partial\colon F\to \fm F\subseteq F$; cf.\ \cite{HI}. With this perspective we see that $\lin(F)$ does not actually depend on the additional grading on $F$.

We note that there is a natural augmentation map $F_{\langle d \rangle}=F_{\langle \leqslant d\rangle}/F_{\langle \leqslant d-1\rangle} \to M_{\langle \leqslant d\rangle}/M_{\langle \leqslant d-1\rangle}$, and that these assemble into an augmentation $\lin(F)\to \bigoplus_d M_{\langle \leqslant d\rangle}/M_{\langle \leqslant d-1\rangle}$; we also note that the target of this map is none other than the associated graded module $\gr(M)$ of $M$ with respect to its $\fm$-adic filtration. Put in other words, there is a natural map
\begin{equation}\label{eq_H0lin}
H_0(\lin(F))\to \gr(M)
\end{equation}
and it may fail to be an isomorphism (but see \cref{thm_complin_steps}).

\subsection*{Koszul modules} Following \cite{HI}, a finitely generated $S$-module $M$ is called \emph{Koszul} if it has a minimal free resolution $F\to M$ such that $\lin(F)$ is acyclic. As shown in~\cite[Remark~1.10]{HI}, $S$ is a Koszul algebra in the classical sense if and only if the $S$-module $k$ is Koszul in the above sense. The connection between linearity and Koszulity was established by R\"omer:

\begin{theorem}[R\"omer {\cite{R}, see also \cite[Theorem~5.6]{IR}}] \label{thm_complin_iff_Kos}
Assume that $S$ is Koszul, and let $M$ be a finitely generated graded $S$-module. Then $M$ is componentwise linear if and only if $M$ is a Koszul module. \qed
\end{theorem}

We will establish a more general analogue of this theorem for componentwise linearity in the first $p$ steps; see \cref{thm_complin_steps} below.

\subsection*{The quasi-linear strand} 
We are interested in the part of the resolution of $M$ that can be reached from minimal generators of $M$ by following linear terms of the differential, and these linear terms in the differential can be detected by the action of $\Ext_S^1(k,k)$ on $\Ext_S(M,k)$. If $\Ext_S(M,k)$ is generated in degree zero under the action of the subalgebra of $\Ext_S(k,k)$ generated by $\Ext_S^1(k,k)$, then $M$ is called a \emph{quasi-Koszul} $S$-module \cite{GM} (this is also called a $\mathcal{K}_1$ module in~\cite{CS}). 

If $M$ is generated in a single degree, then it is Koszul exactly when it is quasi-Koszul by~\cite[Proposition~4.4]{CS}. In general, all Koszul modules are quasi-Koszul (we include a short proof for Stanley--Reisner ideals in \cref{lem_Kos_qKos}, and we prove a generalisation in \cref{prop_qlinstrand}), but the converse does not hold; see \cite[Example~7.2]{CS} for a counterexample (compare as well with \cref{ex_AQK}).

The notion of a quasi-Koszul module motivates the next definition.

\begin{definition} \label{def_strand}
The \emph{quasi-linear strand} of $M$ is the submodule of $\Ext_S(M,k)$ generated by $\Ext_S^0(M,k)$ under the action of the $k$-subalgebra  of $\Ext_S(k,k)$ generated by $\Ext_S^1(k,k)$. Note that if $S$ is a Koszul algebra (as it will always be for us in our applications), then $\Ext_S(k,k)$ is generated in degree one, and the quasi-linear strand of $M$ is simply the $\Ext_S(k,k)$-submodule of $\Ext_S(M,k)$ generated by degree zero elements.
\end{definition}

Clearly $M$ is a quasi-Koszul module if and only if its quasi-linear strand equals all of $\Ext_S(M,k)$. If $M$ is generated in a single degree $d$, then $\beta_{i,j}(M)=0$ for $j<i+d$, and the quasi-linear strand agrees with the ``lowest linear strand" $\bigoplus_{i\geqslant 0} \Ext_S^i(M,k)_{i+d}$

We will often use this definition when $M=I$ is an ideal of $S$. In that case, $\Ext_S^0(I,k)=\Ext_S^1(S/I,k)$, and so we will refer to the $\Ext_S(k,k)$-submodule of $\Ext_S(S/I,k)$ generated by $\Ext_S^1(S/I,k)$ as the \emph{quasi-linear strand} of $S/I$.

\subsection*{The situation for Stanley--Reisner ideals} For the rest of this section we take $S$ to be the standard graded polynomial ring $k[v_1,\ldots,v_m]$. Let $K$ be a simplicial complex on the vertex set $[m]$ with Stanley--Reisner ring $k[K]=S/I_K$. 

By a well-known result of Eagon and Reiner, the graded $S$-module $I_K$ has a linear resolution if and only if $K^\vee$ is \emph{Cohen--Macaulay} over $k$, that is, $k[K^\vee]$ is a Cohen--Macaulay ring~\cite[Theorem~3]{ER}. Building on this work, Herzog and Hibi proved that $I_K$ is componentwise linear if and only if $K^\vee$ is \emph{sequentially Cohen--Macaulay} over $k$, that is, $k[K^\vee(q)]$ is a Cohen--Macaulay ring for every~$q$~\cite[Theorem 2.1]{HH}. A different generalisation of the Eagon--Reiner theorem due to Terai and Yanagawa states that $I_K$ satisfies condition $N_{d,p}$ if and only if $k[K^\vee]$ satisfies Serre's condition $(S_p)$ \cite[Corollary~3.7]{Y}.

By the results of \cite{CS} mentioned above, componentwise linear Stanley--Reisner ideals are strictly contained in the class of ideals $I_K$ that are quasi-Koszul $S$-modules. This larger class can be characterised combinatorially in terms of the Hochster formula. We use the following terminology, following Beben and Grbi\'c~\cite{BG}, who introduced an equivalent definition in the $k=\mathbb{Z}$ case.

\begin{definition} \label{def_HMF}
A simplicial complex $K$ on the vertex set $[m]$ is called a \emph{homology missing face complex} (or \emph{HMF complex}) over $k$ if, for each $U\subseteq[m]$, $\widetilde{H}_*(K_U;k)$ is generated as a $k$-module by missing faces of $K$, that is, the images in homology of the inclusions $K_I \hookrightarrow K_U$ for minimal non-faces $I\notin K$.  
\end{definition}

\begin{proposition} \label{prop_HMF}
Let $K$ be a simplicial complex. The Stanley--Reisner ideal $I_K$ is a quasi-Koszul $S$-module if and only if $K$ is an HMF complex over $k$. 
\end{proposition}

\begin{proof}
This follows immediately from \cref{prop_homology_Z_K} below together with the Hochster formula interpretation of the quasi-linear strand given in \cref{rem_Hoch_strand}. 
\end{proof}

Although HMF complexes $K\neq \partial\Delta^{m-1}$ can essentially never be Gorenstein (cf.\ \cref{sec_manifolds}), we establish an analogue of \cref{prop_HMF} in the Gorenstein setting in \cref{prop_almost_HMF}.

\section{Homology of moment-angle complexes and the sweep action} \label{sec_MACs}

Throughout, $K$ continues to denote an abstract simplicial complex on the vertex set $[m]=\{1,\ldots,m\}$. We assume that $K$ has no ghost vertices; that is, $\{i\}\in K$ for all $i\in[m]$.

\subsection*{Moment-angle complexes and Davis--Januszkiewicz spaces}

For concreteness we write $D^2=\{z\in\mathbb{C} : |z| \leqslant 1\}$ and $S^1=\{z\in\mathbb{C} : |z| = 1\}$, both based at $1\in S^1\subseteq D^2$. The \emph{moment-angle complex} over $K$ is the subspace of $(D^2)^m$ defined by
\[
\mathcal{Z}_K = \bigcup_{I\in K} (D^2,S^1)^I\quad\text{where}\quad (D^2,S^1)^I=\left\{ (z_1,\ldots,z_m)\in (D^2)^m \,:\, z_i\in S^1\text{ if }i\notin I\right\}.
\]
The coordinatewise action of the torus $T^m=(S^1)^m$ on $(D^2)^m$ restricts to an action
\[
a\colon T^m\times\mathcal{Z}_K \to \mathcal{Z}_K, \quad \big((t_1,\ldots,t_m),(z_1,\ldots,z_m)\big) \mapsto (t_1z_1,\ldots,t_mz_m)
\]
called the standard torus action on $\mathcal{Z}_K$. The Borel construction (or homotopy quotient) of this action can be identified with the \emph{Davis--Januszkiewicz space}
\[
DJ_K = \bigcup_{I\in K} (\mathbb{C}P^\infty,*)^I\quad\text{where}\quad (\mathbb{C}P^\infty,*)^I=\left\{ (z_1,\ldots,z_m)\in (\mathbb{C}P^\infty)^m \,:\, z_i=*\text{ if }i\notin I\right\}.
\]
More precisely, $DJ_K$ is a subspace of the classifying space $BT^m=(\mathbb{C}P^\infty)^m$, and the inclusion has homotopy fibre $\mathcal{Z}_K$, yielding a homotopy fibration
\begin{equation} \label{fib}
\mathcal{Z}_K \stackrel{\omega}{\longrightarrow} DJ_K \longrightarrow BT^m
\end{equation}
which can be identified up to homotopy with the Borel fibration of the standard torus action on $\mathcal{Z}_K$; see~\cite[Theorem~4.3.2]{BP}.

The polydisk $(D^2)^m$ decomposes into cells indexed by pairs of disjoint subsets $I,J\subseteq [m]$:
\[
\varkappa(I,J) = \big\{ (z_1,\ldots,z_m) : z_i\in D^2\smallsetminus S^1 \text{ for } i\in 
I,\  z_j\in S^1\smallsetminus\{1\}\text{ for } j\in J,\ 
z_k=1 \text{ for } k\notin I\cup J \big\}.
\]
As described in~\cite[\S 4.4]{BP}, $\varkappa(I,J)$ is an open cell of dimension $2|I|+|J|$ and the moment-angle complex is the cellular subcomplex of $(D^2)^m$ consisting of those cells $\varkappa(I,J)$ with $I\in K$ and $I\cap J=\varnothing$.

Regarding $T^m$ as a cellular subcomplex of $(D^2)^m$ as well, the cellular chain complex $C_*^\mathrm{cw}(T^m;k)$ and $H_*(T^m;k)$ can both be identified with the exterior algebra $\Lambda=\Lambda(\lambda_1,\ldots,\lambda_m)$ with $|\lambda_i|=1$. Since the standard torus action $a\colon T^m\times\mathcal{Z}_K \to \mathcal{Z}_K$ is a cellular map with respect to these cellular structures, it yields an action of $\Lambda$ on $C_*^\mathrm{cw}(\mathcal{Z}_K;k)$, called the \emph{sweep action}, making $C_*^\mathrm{cw}(\mathcal{Z}_K;k)$ into a differential graded $\Lambda$-module, and inducing a graded $\Lambda$-module structure on $H_*(\mathcal{Z}_K;k)$.

We next describe a convenient algebraic model for the homology of $\mathcal{Z}_K$ with this $\Lambda$-module structure. We refer to~\cite[Section~3]{AB} for a thorough study of dg $\Lambda$-module models for~$\mathcal{Z}_K$.

\subsection*{The co-Koszul complex} While we are interested in the homological algebra of the Stanley--Reisner ring $k[K]=S/I_K$, it will be convienient to work with its graded dual, the \emph{Stanley--Reisner coalgebra} $k\langle K\rangle$, which can be more intrinsically defined as the $k$-linear span of symbols $v_\sigma$ indexed by multidegrees $\sigma=(\sigma_1,\ldots,\sigma_m)\in\mathbb{N}^m$ supported on faces of $K$ (less intrinsically, $v_\sigma $ is dual to the monomial $v^\sigma$ in $S$). To be precise,
\[
k\langle K\rangle = \operatorname{span}\Big\{v_\sigma\,:\, \{i\in [m] \,:\, \sigma_i\neq 0\}\in K\Big\},
\]
and $v_\sigma$ is given degree $|v_\sigma|=2\sum_{i=1}^m\sigma_i$.

The co-Koszul complex
\[
\big( k\langle K\rangle \otimes \Lambda(\lambda_1,\ldots,\lambda_m),\; \partial \big), \quad |\lambda_i|=1
\]
can be thought of as the linear dual of the usual Koszul complex of the Stanley--Reisner ring; its differential is given by
\[ 
\partial(v_\sigma\otimes \lambda_J) = \sum_{\sigma_i\neq 0}v_{\sigma-e_i}\otimes \lambda_i\wedge \lambda_J,
\]
where $\lambda_J=\lambda_{j_1}\wedge\cdots\wedge \lambda_{j_r}$ for $J=\{j_1<\cdots <j_r\}\subseteq [m]$, and $v_\varnothing=1$. 

Next, let $R_*(K)$ be the quotient of the co-Koszul complex by all elements $v_\sigma \lambda_J \coloneqq v_\sigma \otimes \lambda_J$ of non-squarefree multidegree. In particular, if we identify squarefree multidegrees $\sigma\in\{0,1\}^m$ with subsets $I=\{i\in [m] \,:\, \sigma_i=1\}$ of $[m]$ and write  $v_I=v_\sigma$, then
\[
R_*(K) = \operatorname{span}\{ v_I\lambda_J \,:\, I\in K,\, I\cap J=\varnothing \}.
\]
We introduce a bigrading on $R_*(K)$, defining $R^i(K)_U \subseteq R_{2|U|-i}(K)$ for each $i\geqslant 0$ and $U\subseteq [m]$ by setting
\begin{align*}
R^i(K)_U &= \big( k\langle K\rangle \otimes \Lambda^i(\lambda_1,\ldots,\lambda_m) \big)_U =\operatorname{span}\{ v_I\lambda_J \,:\, |J|=i,\, I\sqcup J=U \}.
\end{align*}
Note that $\partial(R_i(K))\subseteq R_{i-1}(K)$ while  $\partial(R^i(K)_U) \subseteq R^{i+1}(K)_U$.

We also point out that $R_*(K)$ inherits a dg $\Lambda$-module structure from the co-Koszul complex, explicitly given by
\begin{align*}
    \Lambda\otimes R_*(K) &\longrightarrow R_*(K) \\
    \lambda_j\otimes v_I\lambda_J\; &\longmapsto \begin{cases} v_I\otimes \lambda_j\wedge \lambda_J & \text{if } j\notin I \\ 0 & \text{if } j\in I. \end{cases}
\end{align*}
This structure models the sweep action on the homology of the moment-angle complex: there is a natural isomorphism of dg $\Lambda$-modules $R_*(K)\cong C^\mathrm{cw}_*(\mathcal{Z}_K;k)$; see~\cref{prop_homology_Z_K} below. 

\begin{remark}
Through Yoneda composition the cohomology $H^*(R_*(X))=\Ext_S(k[K],k)$ is naturally a left $\Lambda\cong \Ext_S(k,k)$-module, and this action coincides with the module structure induced from the just-defined  action of $\Lambda$ on $R_*(K)$. This can be seen as follows: the Koszul complex $\mathbb{K}= S\otimes \Lambda^*$ is a resolution of $k$ and a dg  $\Lambda$-module. There are quasi-isomorphisms of dg $\Lambda$-modules
\[
R_*(K)\twoheadleftarrow (k[K]\otimes_S \mathbb{K})^*={\rm Hom}_k(k[K]\otimes_S \mathbb{K},k) \cong {\rm Hom}_S(k[K], \mathbb{K}^*). 
\]
Moreover the action of $\Lambda$ on $\mathbb{K^*}$ induces an algebra isomorphism $\Lambda \xrightarrow{\cong} H^*({\rm End}_S(\mathbb{K}^*))= \Ext_S(k,k)$.
The compatibility of these isomorphisms establishes the claim of the remark.
\end{remark}

\subsection*{Hochster's decomposition and the sweep action} According to Hochster's formula~\cite{H} or~\cite[Theorem~3.2.9]{BP}, for each $U\subseteq [m]$ there is an isomorphism of chain complexes
\[ 
\begin{aligned}
R^*(K)_U &\longrightarrow \widetilde{C}_{|U|+*-1}(K_U;k) \\
v_I\lambda_{U\smallsetminus I} &\longmapsto (-1)^{\varepsilon(I,U)} I,
\end{aligned}
\]
where $\varepsilon(I,U)=|\{(i,u)\in I\times U : i>u\}|$, and $\widetilde{C}_{-1}(\varnothing;k)=k$ by convention. By summing over all $U\subseteq [m]$ and passing to homology we obtain the (dual of the) usual Hochster decomposition
\[ 
H_*(\mathcal{Z}_K;k) \cong \Ext_S(k[K],k) \cong \bigoplus_{U\subseteq [m]} \widetilde{H}_*(K_U;k).
\]
We next describe the $\Lambda(\lambda_1,\ldots,\lambda_m)$-module structure on $H_*(\mathcal{Z}_K)$ in terms of this decomposition. Define for each $j\in [m]\smallsetminus U$ the following map, induced by the inclusion $K_U \hookrightarrow K_{U\cup\{j\}}$ up to the indicated sign:
\begin{align*}
 \lambda_j\colon \widetilde{C}_n(K_U)  
&\longrightarrow\widetilde{C}_n(K_{U\cup\{j\}})\\
\qquad(I\in K_U) &\longmapsto (-1)^{\varepsilon(j,U)+n+1}I.
\end{align*}

When $j\in U$ we define the action of $\lambda_j$ to be zero on $\widetilde{C}_*(K_U;k)$. These are anti-commuting chain maps, yielding a dg $\Lambda=\Lambda(\lambda_1,\ldots,\lambda_m)$-module structure on $\bigoplus_{U\subseteq[m]}\widetilde{C}_*(K_U;k)$, and moreover they are compatible with the dg $\Lambda$-module structures already defined on $C_*^{\mathrm{cw}}(\mathcal{Z}_K;k)$ and $R_*(K)$; all of this is contained in the next proposition.

\begin{proposition} \label{prop_homology_Z_K}
Let $K$ be a simplicial complex on vertex set $[m]$. Then there are isomorphisms of differential graded $\Lambda(\lambda_1,\ldots,\lambda_m)$-modules
\begin{align*}
C_*^{\mathrm{cw}}(\mathcal{Z}_K;k) &\longleftarrow R_*(K) \longrightarrow \bigoplus_{U\subseteq [m]}\widetilde{C}_*(K_U;k) \\
C_{2|I|+|J|}^{\mathrm{cw}}(\mathcal{Z}_K;k) \ni\varkappa(I,J) &\longleftarrow\joinrel\mapstochar \;\,v_I\lambda_J\, \longmapsto (-1)^{\varepsilon(I,I\cup J)}I \in \widetilde{C}_{|I|-1}(K_{I\cup J};k)
\end{align*}
inducing isomorphisms of graded $\Lambda(\lambda_1,\ldots,\lambda_m)$-modules
\[
H_*(\mathcal{Z}_K;k) \cong \Ext_S(k[K],k) \cong \bigoplus_{U\subseteq [m]}\widetilde{H}_*(K_U;k).
\]
\end{proposition}

\begin{proof} 
Beginning with the left-hand map, we observe that the inverse mapping $\varkappa(I,J) \mapsto v_I\lambda_J$ is precisely the $k$-linear dual of the natural isomorphism of dg $\Lambda$-modules established in~\cite[Lemma~3.3]{AB}. Likewise, that the right-hand map $v_I\lambda_J \mapsto (-1)^{\varepsilon(I,I\cup J)}I$ is an isomorphism of complexes is~\cite[Theorem~3.2.9]{BP}, and the fact that it is a map of $\Lambda$-modules follows by taking the linear duals of all complexes and maps in~\cite[Lemma~3.5]{AB}.
\end{proof}

\begin{remark} \label{rem_Hoch_strand}
Using \cref{prop_homology_Z_K}, the quasi-linear strand of $k[K]$ admits simple descriptions in terms of the Hochster formula and the co-Koszul complex. Under the isomorphism of $\Lambda$-modules $\Ext_S(k[K],k)\cong\bigoplus_{U\subseteq[m]}\widetilde{H}_*(K_U;k)$, the quasi-linear strand corresponds to the subspace of $\bigoplus_{U\subseteq[m]}\widetilde{H}_*(K_U;k)$ spanned by the images of the natural inclusion maps 
\[
\widetilde{H}_*(K_I;k)=\widetilde{H}_*(\partial\Delta^{|I|-1};k) \longrightarrow \widetilde{H}_*(K_{I\cup\{j_1,\ldots,j_r\}};k)
\]
for all minimal non-faces $I\notin K$ and $j_1,\ldots,j_r\in [m]\smallsetminus I$. In terms of the co-Koszul complex computing $\Ext_S(k[K],k)$, these generators of the quasi-linear strand are represented by cycles
\[
\left( \sum_{i\in I} v_{I\smallsetminus i}\lambda_i \right)\lambda_{j_1}\!\cdots \lambda_{j_r} \in R^{r+1}(K)_{I\cup\{j_1,\ldots,j_r\}}
\]
in $R_{2|I|+r-1}(K) \cong C^\mathrm{cw}_{2|I|+r-1}(\mathcal{Z}_K;k)$.
\end{remark}

\subsection*{Cohomology operations and the minimal resolution of the Stanley--Reisner ring} 

We finish this section by recalling a result from~\cite{AB} relating the exterior algebra action on the cohomology of a moment-angle complex $\mathcal{Z}_K$ to the linear part of the minimal free resolution of the corresponding Stanley--Reisner ring $k[K]$. Following the notation of~\cite{AB}, we consider the cohomology ring $H^*(\mathcal{Z}_K)$ as a nonnegatively graded module over an exterior algebra $\Lambda(\iota_1,\ldots,\iota_m)$ with generators of degree $-1$ by dualising the $\Lambda(\lambda_1,\ldots,\lambda_m)$-module structure in homology given by the sweep action. Then the maps $\iota_j\colon H^*(\mathcal{Z}_K)\to H^{*-1}(\mathcal{Z}_K)$ are derivations which are the primary operations $\delta_{\{j\}}=\iota_j$ in a system of \emph{higher cohomology operations} 
\[
\big\{\delta_J\in\operatorname{End}H^*(\mathcal{Z}_K) \,:\, J\subseteq[m]\big\}, \qquad |\delta_J|=1-2|J|,
\]
in the sense of~Goresky--Kottwitz--MacPherson~\cite{GKM}, induced by the $T^m$-action on $\mathcal{Z}_K$. With respect to the $\mathbb{Z}\times\mathbb{Z}^m$-grading on 
\[
H^*(\mathcal{Z}_K;k)=\bigoplus_{2|U|-i=*}H^{-i,U}(\mathcal{Z}_K;k), \qquad H^{-i,U}(\mathcal{Z}_K;k)\cong\Tor^S_{i}(k[K],k)_U
\]
(cf.~\cite[Theorem~4.5.7]{BP}), each $\delta_J$ defines maps $H^{-i,U}(\mathcal{Z}_K;k) \to H^{-i+1,U\smallsetminus J}(\mathcal{Z}_K;k)$.

\begin{theorem}[{\cite[Theorem~A]{AB}}] \label{thm_min_res}
The minimal graded free resolution of the Stanley--Reisner ring $k[K]$ is isomorphic to
\[
\big(S\otimes_k H^{*,*}(\mathcal{Z}_K;k),\; \delta\big), \qquad \delta=\sum_{J\subseteq[m]} v_J\otimes\delta_J. 
\]
In particular, the linear part of the resolution is given by
$\left(S\otimes_k H^{*,*}(\mathcal{Z}_K;k),\; \delta^\mathrm{lin}\right)$ with differential
\[
\pushQED{\qed} 
\delta^\mathrm{lin}=\sum_{j=1}^m v_j\otimes\iota_j. \qedhere
\popQED
\]
\end{theorem} 

\begin{remark}\label{rem_linear_part_action}
It follows from \cref{thm_min_res} that the data of $H_*(\mathcal{Z}_K;k)$ as a $\Lambda$-module is equivalent to the data of the linear part of the minimal free resolution of $k[K]$. Interpreted in terms of the Bernstein--Gel'fand--Gel'fand (BGG) correspondence, the second part of the theorem states that the BGG dual of the dg $\Lambda$-module $H_*(\mathcal{Z}_K;k)$ (with zero differential) is precisely the linear part of the minimal free resolution of $k[K]$.
\end{remark}

\begin{lemma} \label{lem_Kos_qKos}
If $I_K$ is a Koszul $S$-module, then $I_K$ is a quasi-Koszul $S$-module.
\end{lemma}

\begin{proof}
Since the minimal free resolutions of $I_K$ and $k[K]=S/I_K$ only differ by a degree shift, the assumption and \cref{thm_min_res} imply that $\left(S\otimes_k H^{*,*}(\mathcal{Z}_K;k),\; \delta^\mathrm{lin}\right)$ is exact above homological degree one.  Let $\alpha \in H^{-i,*}(\mathcal{Z}_K;k)$ with $i>1$, and note that $1\otimes\alpha \in S\otimes_k H^{*,*}(\mathcal{Z}_K;k)$ is not in the image of $\delta^\mathrm{lin}$ by minimality ($\operatorname{im}\delta\subset \mathfrak{m}\otimes_k H^*(\mathcal{Z}_K;k)$). Therefore $1\otimes\alpha \notin \ker\delta^\mathrm{lin}$ by exactness, so $\iota_j(\alpha)$ is nonzero for some $j\in[m]$. Dualising, since $\iota_j$ is dual to $\lambda_j\colon H_*(\mathcal{Z}_K;k)\to H_{*+1}(\mathcal{Z}_K;k)$, this shows that every class in $H_{-i,*}(\mathcal{Z}_K;k) \cong \Ext_S^i(k[K],k)$ with $i>1$ is in the span of the images of the $\lambda_j$, and hence contained in the quasi-linear strand of $k[K]$. Therefore $I_K$ is quasi-Koszul.
\end{proof}

\section{The Hurewicz image and the quasi-linear strand} \label{sec_hurewicz}

\subsection*{The sweep action on the Hurewicz image}

The $i^\text{th}$ coordinate circle action $a_i\colon S^1\times \mathcal{Z}_K \to \mathcal{Z}_K$ is obtained by restricting the standard action $a\colon T^m\times\mathcal{Z}_K\to\mathcal{Z}_K$. Since we assume that $\{i\}\in K$ for all $i\in[m]$, there is a  canonical extension of $a_i$ given by
\[
\bar{a}_i\colon \mathcal{Z}_K\times S^1 \cup *\times D^2 \longrightarrow \mathcal{Z}_K, \quad \big((z_1,\ldots,z_m),z\big) \mapsto (z_1,\ldots,z_iz,\ldots,z_m). 
\]
Recall that the \emph{half-smash product} $X\ltimes Y$ of pointed spaces $X$ and $Y$ is defined as the quotient space $ X\ltimes Y = {X\times Y}/{X\times *}$. Since the inclusion $S^1\times * \to S^1\times\mathcal{Z}_K$ is a cofibration,  the half-smash product $S^1\ltimes\mathcal{Z}_K$ is homotopy equivalent to the mapping cone of $S^1\times * \to S^1\times\mathcal{Z}_K$. In other words, there is a homotopy equivalence
\[
S^1\ltimes\mathcal{Z}_K\simeq \mathcal{Z}_K\times S^1 \cup *\times D^2.
\]
Making this identification, we note that $\bar{a}_i$ represents the (unique up to homotopy) extension of $a_i$ from $S^1\times\mathcal{Z}_K$ to $S^1\ltimes\mathcal{Z}_K$.

By definition the Davis--Januszkiewicz space $DJ_K$ is a subspace of the cartesian product $(\mathbb{C}P^\infty)^m$ containing the $m$-fold wedge $(\mathbb{C}P^\infty)^{\vee m}$ (see \cref{sec_MACs}). For each $i\in[m]$, let
\[
\mu_i\colon S^2 \longrightarrow (\mathbb{C}P^\infty)^{\vee m} \lhook\joinrel\longrightarrow DJ_K 
\]
denote the canonical generator of  $\pi_2(DJ_K)\cong\mathbb{Z}^m$ defined using the inclusion of the bottom cell of the $i^\text{th}$ wedge summand of $(\mathbb{C}P^\infty)^{\vee m}$.

\begin{lemma} \label{comm_square}
With the notation just introduced, there is a homotopy commutative diagram
\[
\begin{tikzcd}
\mathcal{Z}_K\times S^1 \cup *\times D^2 \ar[r,"\bar{a}_i"] \ar[d,"\omega\times q"] & \mathcal{Z}_K \ar[d,"\omega"] \\
DJ_K\vee S^2 \ar[r,"1\vee\mu_i"] & DJ_K,
\end{tikzcd}
\]
where $q\colon D^2\to S^2$ is the pinch map collapsing $S^1$ to a point. 
\end{lemma}

\begin{proof}
Since $T^m\to \mathcal{Z}_K\xrightarrow{\omega} DJ_K$ is a principal fibration (with classifying map the inclusion $DJ_K\to \prod_{i=1}^m\mathbb{C}P^\infty=BT^m$), the action of $T^m$ on $\mathcal{Z}_K$ fits into a homotopy commutative diagram
\[
\begin{tikzcd}
T^m\times\mathcal{Z}_K \ar[r,"a"] \ar[d,"\pi_2"] & \mathcal{Z}_K \ar[d,"\omega"] \\
\mathcal{Z}_K \ar[r,"\omega"] & DJ_K.
\end{tikzcd}
\]
It follows that the diagram in the statement of the lemma commutes up to homotopy on $\mathcal{Z}_K\times S^1$ since the composite $\mathcal{Z}_K\times S^1 \xrightarrow{\omega\times q} DJ_K\vee S^2 \xrightarrow{1\vee\mu_i} DJ_K$ is equal to the composite $\mathcal{Z}_K\times S^1 \xrightarrow{\pi_1} \mathcal{Z}_K \xrightarrow{\omega} DJ_K$. The diagram strictly commutes on $*\times D^2$ since $\bar{a}_i|_{*\times D^2}$ coincides with the $i^\text{th}$ coordinate inclusion $D^2 \to \mathcal{Z}_K\subseteq \prod_{i=1}^m D^2$, and $\omega$ is the map induced by $D^2 \stackrel{q}{\longrightarrow} S^2\lhook\joinrel\longrightarrow \mathbb{C}P^\infty$. By the homotopy extension property, a homotopy realising the homotopy commutativity of the restriction of the diagram to $\mathcal{Z}_K\times S^1$ can be extended to $\mathcal{Z}_K\times S^1 \cup *\times D^2$.
\end{proof}

For a pointed topological space $X$, we denote the Hurewicz map by $h\colon \pi_n(X)\to H_n(X)$.

\begin{theorem} \label{thm_Hurewicz_equiv}
Let $K$ be a simplicial complex on the vertex set $[m]$. If $f\in \pi_n(\mathcal{Z}_K)$ has Hurewicz image $h(f) \in H_n(\mathcal{Z}_K)$, then 
\[ 
\lambda_ih(f) = h\big([\overline{\mu_i,\omega\circ f}]\big) \in H_{n+1}(\mathcal{Z}_K) 
\]
for all $i\in[m]$, where $[\overline{\mu_i,\omega\circ f}]$ is the unique lift of the Whitehead product $[\mu_i,\omega\circ f] \in \pi_{n+1}(DJ_K)$.
In particular, the Hurewicz image of $\mathcal{Z}_K$ is closed under the sweep action.
\end{theorem}

\begin{proof}
Let $f\in \pi_n(\mathcal{Z}_K)$ and let $i\in [m]$. We will construct an explicit representative of the unique lift $[\overline{\mu_i,\omega\circ f}] \in \pi_{n+1}(\mathcal{Z}_K)$ of $[\mu_i,\omega\circ f] \in \pi_{n+1}(DJ_K)$ through the map $\omega\colon \mathcal{Z}_K \to DJ_K$. It is well known that the universal Whitehead product $[\iota_n,\iota_2]\colon S^{n+1} \to S^n\vee S^2$ factors as a composite
\[ S^{n+1}\cong \partial(D^n\times D^2)= (D^n\times S^1)\cup (S^{n-1}\times D^2) \xrightarrow{q\times q} (S^n\times *)\cup(*\times S^2)=S^n\vee S^2, \]
where $q$ is the pinch map collapsing the boundary of a disk to a point. We therefore have a homotopy commutative diagram
\[
\begin{tikzcd}[row sep=large]
(D^n\times S^1)\cup (S^{n-1}\times D^2) \ar[r,"q\times 1"] \ar[d,transform canvas={xshift=0.2ex},-] \ar[d,transform canvas={xshift=-0.4ex},-] & (S^n\times S^1)\cup (*\times D^2) \ar[r,"f\times 1"] \ar[d,"1\times q"] & (\mathcal{Z}_K\times S^1)\cup (*\times D^2) \ar[r,"\bar{a}_i"] \ar[d,"\omega\times q"] & [-5pt] \mathcal{Z}_K \ar[d,"\omega"] \\
\partial(D^n\times D^2)\cong S^{n+1} \ar[r,"{[\iota_n,\iota_2]}"] & S^n\vee S^2 \ar[r,"(\omega\circ f)\vee 1"] & DJ_K\vee S^2 \ar[r,"1\vee\mu_i"] & DJ_K
\end{tikzcd}
\]
where the square on the right commutes up to homotopy by~\cref{comm_square}. The composite along the bottom row is $[\omega\circ f,\mu_i]=[\mu_i,\omega\circ f]$ by naturality of the Whitehead product. The composite along the top row is therefore the unique lift $[\overline{\mu_i,\omega\circ f}]$ through $\omega$. We next use this factorisation to compute the Hurewicz image of $[\overline{\mu_i,\omega\circ f}]$.

Viewing $(S^n\times S^1)\cup(*\times D^2)$ as a pushout over $*\times S^1$, it is clearly homotopy equivalent to the cofibre of the inclusion $S^1 \to S^n\times S^1$. We therefore have
\[
(S^n\times S^1)\cup(*\times D^2) \simeq S^n \rtimes S^1 \simeq S^n\vee S^{n+1}
\]
(see \cite[Proposition~7.7.8]{S} for the second equivalence), and it is straightforward to check that the first map $q\times 1$ along the top row of the diagram is given by the inclusion $S^{n+1} \to S^n\vee S^{n+1}$ of the second wedge summand up to homotopy. The second map $f\times 1$ along the top row can be identified up to homotopy with the map $f\rtimes 1\colon S^n\rtimes S^1 \to \mathcal{Z}_K\rtimes S^1$. The Hurewicz image of the composite $(f\times 1)\circ(q\times 1)$ is therefore $h(f)\otimes [S^1] \in H_{n+1}(\mathcal{Z}_K\rtimes S^1)$, where $[S^1]$ denotes the fundamental class of $S^1$ and, by abuse of notation, $h(f)\otimes [S^1]$ denotes the image of this class under the quotient map $\mathcal{Z}_K\times S^1 \to \mathcal{Z}_K\rtimes S^1$. Lastly, 
\[
(\bar{a}_i)_*\big(h(f)\otimes [S^1]\big)=(a_i)_*\big(h(f)\otimes [S^1]\big)=\lambda_ih(f)
\] 
by definition of the sweep action $\Lambda(\lambda_1,\ldots,\lambda_m) \otimes H_*(\mathcal{Z}_K) \to H_*(\mathcal{Z}_K)$ induced by the standard $T^m$-action on $\mathcal{Z}_K$. This proves the desired equality.
\end{proof}

\begin{remark}
We point out that the Hurewicz image is not generally invariant under the sweep action for arbitrary torus actions, as $T^2$ acting on itself shows. This example also illustrates the necessity of our global assumption that $K$ has no ghost vertices (since $\mathcal{Z}_K=T^2$ for the irrelevant complex $K=\{\varnothing\}$ on the vertex set $[2]$).
\end{remark}

\subsection*{Sphericity of the quasi-linear strand} 

Recall from \cref{def_strand} that the quasi-linear strand of the Stanley--Reisner ring $k[K]$ is the $\Ext_S(k,k)$-submodule of $\Ext_S(k[K],k)$ generated by the cohomological degree one part $\Ext_S^1(k[K],k)$. Using \cref{prop_homology_Z_K}, we identify this with the $\Lambda$-submodule of $H_*(\mathcal{Z}_K;k)$ generated by $H_{-1,*}(\mathcal{Z}_K;k)$. 

It follows from \cref{thm_Hurewicz_equiv} that every homology class in the quasi-linear strand is spherical: 

\begin{theorem} \label{thm_Hurewicz_strand}
The quasi-linear strand of $k[K]$ is contained in (the $k$-linear span of) the image of the Hurewicz map \[ h\colon \pi_*(\mathcal{Z}_K)\longrightarrow H_*(\mathcal{Z}_K;k). \]
\end{theorem}

\begin{proof}
The cohomological degree one part $H_{-1,*}(\mathcal{Z}_K;k)\cong\Ext_S^1(k[K],k)$ is contained in the image of the Hurewicz map since for each minimal non-face $I\notin K$, the corresponding homology class (represented by the cycle $\sum_{i\in I}v_{I\smallsetminus i}\lambda_i$ in the co-Koszul complex $R_*(K)$) is the Hurewicz image of the element $\omega_I\in\pi_{2|I|-1}(\mathcal{Z}_K)$ given by the natural map $\omega_I\colon S^{2|I|-1}=\mathcal{Z}_{K_I} \to \mathcal{Z}_K$ induced by the inclusion $K_I\hookrightarrow K$. By \cref{thm_Hurewicz_equiv}, the Hurewicz image is closed under the sweep action and therefore spans the entire quasi-linear strand, that is, the image of $\Lambda\otimes H_{-1,*}(\mathcal{Z}_K;k)\to H_*(\mathcal{Z}_K;k)$.
\end{proof}

Let $\operatorname{sk}_nX$ denote the $n$-skeleton of a CW-complex $X$. The next corollary shows that the Green--Lazarsfeld index of a Stanley--Reisner ideal, $\operatorname{index}(I_K)$, gives a lower bound on the \emph{partial formality} of $\mathcal{Z}_K$ in the sense of~\cite{Ma}. 

\begin{corollary} \label{cor_Ndr}
If $I_K$ satisfies condition $N_{d,r}$ over $\mathbb{F}_p$ for all primes $p$ (respectively, for some $p$), then $\operatorname{sk}_{2d+r}\mathcal{Z}_K$ has the homotopy type (respectively, $p$-local homotopy type) of a wedge of spheres. In particular, $\mathcal{Z}_K$ is $(2d+r-1)$-formal.
\end{corollary}

\begin{proof}
Suppose that $I_K$ is $N_{d,r}$ over $k=\mathbb{F}_p$. Since $I_K$ is generated in degree $d$, $\mathcal{Z}_K$ is $(2d-2)$-connected with $H_{2d-1}(\mathcal{Z}_K;k)\cong\Ext_S^1(k[K],k)$. Since $I_K$ is $N_{d,r}$, it follows that 
\begin{equation} \label{eq_strand}
\Ext_S^{\leqslant r}(k[K],k)=k \oplus\bigoplus_{i=1}^r\Ext_S^i(k[K],k)_{i+d-1} \cong H_{\leqslant 2d+r-2}(\mathcal{Z}_K;k)
\end{equation}
is contained in the linear strand of $k[K]$. To see that $H_{\leqslant 2d+r}(\mathcal{Z}_K;k)$ is also contained in the linear strand, note that in the minimal free resolution $\big( S\otimes_k H^*(\mathcal{Z}_K;k),\, \delta \big)$ with $\delta=\sum_{J\subseteq[m]}v_J\otimes\delta_J$ (see \cref{thm_min_res}), each operation~$\delta_J$ with $|J|>1$ has topological degree $|\delta_J|\leqslant -3$, defining maps $\Tor^S_i(k[K],k)_j \to \Tor^S_{i-1}(k[K],k)_{j-|J|}$. So if some $\alpha \in H^*(\mathcal{Z}_K;k)$ of degree $2d\leqslant *\leqslant 2d+r$ supports such an operation, then $\delta_J(\alpha)\in H^{\leqslant 2d+r-3}(\mathcal{Z}_K;k)$ defines an element of $\Tor^S_i(k[K],k)_{i+d-1}$ for some $i\leqslant r-1$ by~\eqref{eq_strand}, implying that $\alpha$ defines a nonzero element of $\Tor^S_{i+1}(k[K],k)_{>i+d}$ lying outside the (homological) linear strand, contradicting the $N_{d,r}$ condition. Therefore each $\delta_J$ with $|J|>1$ vanishes on $H^*(\mathcal{Z}_K;k)$ in degrees $2d\leqslant *\leqslant 2d+r$. As $1\otimes\alpha \in S\otimes_k H^*(\mathcal{Z}_K;k)$ is not in the image of $\delta$ by minimality ($\operatorname{im}\delta\subset \mathfrak{m}\otimes_k H^*(\mathcal{Z}_K;k)$), exactness of the resolution implies $\delta_{\{j\}}(\alpha)=\iota_j(\alpha) \neq 0$ for some $j\in[m]$. Dualising, every element of $H_*(\mathcal{Z}_K;k)$, $2d\leqslant *\leqslant 2d+r$, is in the span of the images of the $\lambda_j$ maps and therefore part of the linear strand.

\cref{thm_Hurewicz_strand} now implies that the mod $p$ reduction of the Hurewicz map is surjective for $\operatorname{sk}_{2d+r}\mathcal{Z}_K$. Since $\mathcal{Z}_K$ is simply connected, Whitehead's theorem implies that $\operatorname{sk}_{2d+r}\mathcal{Z}_K$ is $p$-locally homotopy equivalent to a wedge of spheres. Similarly, if $I_K$ is $N_{d,r}$ over $\mathbb{F}_p$ for all $p$, then $\operatorname{sk}_{2d+r}\mathcal{Z}_K$ has the homotopy type a wedge of spheres after localisation at any prime. The universal coefficient theorem therefore gives isomorphisms
$H_*(\operatorname{sk}_{2d+r}\mathcal{Z}_K;\mathbb{Z}) \otimes \mathbb{F}_p \to H_*(\operatorname{sk}_{2d+r}\mathcal{Z}_K;\mathbb{F}_p)$ for every prime~$p$, and it follows that the Hurewicz map $h\colon \pi_*(\operatorname{sk}_{2d+r}\mathcal{Z}_K) \to H_*(\operatorname{sk}_{2d+r}\mathcal{Z}_K;\mathbb{Z})$ is surjective since its mod~$p$ reduction is for all primes~$p$.
\end{proof}

In particular, $\mathcal{Z}_K$ is homotopy equivalent to a wedge of spheres (and hence formal) when $I_K$ has a linear resolution, or equivalently $k[K^\vee]$ is Cohen--Macaulay, over all fields $k$. Generalising this, the next two corollaries recover results of Beben--Grbi\'c~\cite[Proposition~3.18]{BG} and Iriye--Kishimoto~\cite[Corollary~1.5]{IK19} concerning HMF complexes and dual sequentially Cohen--Macaulay complexes, respectively.\footnote{We point out here that Iriye and Kishimoto's hypothesis in~\cite[Corollary~1.5]{IK19} that $K$ is dual sequentially Cohen--Macaulay over $\mathbb{Z}$ is equivalent to the hypothesis in \cref{cor_compwise} above that $I_K$ is componentwise linear over $\mathbb{F}_p$ for all primes $p$ (equivalently, over all fields $k$). This can be seen using Reisner's criterion~\cite{Re} in combination with Herzog and Hibi's theorem~\cite[Theorem 2.1]{HH} on the duality between sequential Cohen--Macaulayness and componentwise linearity.}

\begin{corollary} \label{cor_HMF}
If $K$ is an HMF complex over $\mathbb{Z}$ (resp. $\mathbb{F}_p$), then $\mathcal{Z}_K$ has the homotopy type (resp. $p$-local homotopy type) of a wedge of spheres.
\end{corollary}

\begin{proof}
By \cref{prop_HMF} and \cref{thm_Hurewicz_strand}, the Hurewicz map (or its mod $p$ reduction) is surjective in all degrees. Since $\mathcal{Z}_K$ is simply connected, the result follows.
\end{proof}

\begin{corollary} \label{cor_compwise}
If $I_K$ is componentwise linear over $\mathbb{F}_p$ for all primes $p$ (resp. for some $p$), then $\mathcal{Z}_K$ has the homotopy type (resp. $p$-local homotopy type) of a wedge of spheres.
\end{corollary}

\begin{proof}
This follows from \cref{cor_HMF}. To see that $K$ is HMF when $I_K$ is componentwise linear, recall that the former condition holds precisely when $I_K$ is a quasi-Koszul module (\cref{prop_HMF}), while the latter holds precisely when $I_K$ is a Koszul module (\cref{thm_complin_iff_Kos}), and Koszul modules $I_K$ are quasi-Koszul by \cref{lem_Kos_qKos}.
\end{proof}

Recall that $k[K]$ is \emph{Golod} if all products and (higher) Massey products are trivial in $\Tor^S_+(k[K],k)$. \cref{cor_HMF} implies that the Stanley--Reisner rings of HMF complexes are Golod. In particular, we recover the following result of Herzog, Reiner and Welker~\cite{HRW}.

\begin{corollary}
If $I_K$ is componentwise linear over $k$, then $k[K]$ is a Golod ring. \qed
\end{corollary}

\subsection*{The flag case}

The relationship between Golodness and the condition that $\mathcal{Z}_K$ is homotopy equivalent to a wedge of spheres is especially tight when $K$ is a flag complex. The authors of~\cite{GPTW} showed that, among flag complexes, each condition holds
precisely when the $1$-skeleton $K^1$ of $K$ is a chordal graph. In this section we refine this result using a combinatorial characterisation of the $N_{2,r}$ condition due to Eisenbud, Green, Hulek and Popescu~\cite{EGHP}.

Recall that a simplicial complex $K$ is \emph{flag} if each of its minimal non-faces consists of exactly two vertices, or equivalently, the ideal $I_K$ is generated in degree two. A flag complex $K$ is therefore the clique complex of its underlying graph $K^1$. Recall also that a graph is called \emph{chordal} if it contains no chordless cycles (i.e., no cycle of length at least four with no chord).

If $K$ is a flag complex with chordal $1$-skeleton, then $\mathcal{Z}_K$ can be seen to be homotopy equivalent to a wedge of spheres by combining \cref{thm_Hurewicz_strand} with Fr\"oberg's classification of quadratic Stanley--Reisner ideals with linear resolutions: \cite[Theorem~1]{Fro} states that when $K$ is flag, $I_K$ has a linear resolution if and only if $K^1$ is chordal. More generally,~\cite[Theorem~2.1]{EGHP} classifies Stanley--Reisner ideals satisfying condition $N_{2,r}$ by bounding the length of chordless cycles in $K^1$. Combining \cref{thm_Hurewicz_strand} with this refinement of Fr\"oberg's theorem leads to the following.

\begin{theorem} \label{thm_flag}
Let $K$ be a flag complex and let $r>1$. Then the following are equivalent:
\begin{enumerate}[label={\normalfont(\alph*)}]
\item $\operatorname{sk}_{r+4}\mathcal{Z}_K$ is homotopy equivalent to a wedge of spheres;
\item the graph $K^1$ contains no chordless cycles of length $\leqslant r+2$.
\end{enumerate}
\end{theorem}

\begin{proof}
If $K^1$ contains no chordless cycles of length $\leqslant r+2$, then $I_K$ satisfies condition $N_{2,r}$ over any field by~\cite[Theorem~2.1]{EGHP}. Therefore $\operatorname{sk}_{r+4}\mathcal{Z}_K$ has the homotopy type of a wedge of spheres by \cref{cor_Ndr}.

Conversely, suppose that the graph $K^1$ contains a chordless cycle of length $4\leqslant\ell\leqslant r+2$, and let $I\subseteq[m]$ be the set of vertices of this cycle. Then the full subcomplex $K_I$ is the boundary of an $\ell$-gon, and $\mathcal{Z}_{K_I}$ is a retract of $\mathcal{Z}_K$ (since the projection $\prod_{i\in[m]}D^2 \to \prod_{i\in I}D^2$ restricts to a left inverse of the inclusion $\mathcal{Z}_{K_I} \hookrightarrow \mathcal{Z}_K$; cf.~\cite[Lemma~2.2.3]{DS}). Since $\mathcal{Z}_{K_I}$ is $(\ell+2)$-dimensional with a non-trivial cup product in cohomology, $\mathcal{Z}_{K_I} \hookrightarrow \mathcal{Z}_K$ factors through $\operatorname{sk}_{r+4}\mathcal{Z}_K$, implying that the latter does not have the homotopy type of a wedge of spheres.
\end{proof}

\begin{remark}\ 
\begin{enumerate}[label={\normalfont(\arabic*)}]
    \item The $r=\infty$ (or $r\geqslant m-2$) case of \cref{thm_flag} recovers the classification of flag complexes $K$ for which $\mathcal{Z}_K$ is homotopy equivalent to a wedge of spheres in~\cite[Theorem~4.6]{GPTW}.
    \item At each stage of the skeletal filtration of a moment-angle complex
    \[
    \mathrm{pt}=\operatorname{sk}_0\mathcal{Z}_K \subseteq\cdots\cdots\subseteq \operatorname{sk}_n\mathcal{Z}_K \subseteq \operatorname{sk}_{n+1}\mathcal{Z}_K \subseteq\cdots\cdots\subseteq \mathcal{Z}_K,
    \]
    the $(n+1)$-skeleton is the homotopy cofibre of some attaching map $f_n\colon (S^n)^{\vee t_n} \to \operatorname{sk}_n\mathcal{Z}_K$, where $(S^n)^{\vee t_n}=\bigvee_{i=1}^{t_n} S^n$ is a finite wedge of $n$-spheres. 
    If $K$ is flag, then it follows from the proof of \cref{thm_flag} that the first attaching map $f_n$ which is not null-homotopic occurs precisely for $n=\operatorname{index}(I_K)+4$, where $\operatorname{index}(I_K)$ is the Green--Lazarsfeld index of $I_K$. For general $K$, this homotopy invariant of $\mathcal{Z}_K$ is not determined by $\operatorname{index}(I_K)$ alone. 
\end{enumerate}    
\end{remark}

\subsection*{Commutators and the homology suspension}
 
The Pontryagin ring of a Davis--Januszkiewicz space can be identified with the Yoneda algebra of the corresponding Stanley--Reisner ring via an isomorphism of graded algebras $H_*(\Omega DJ_K;k)\cong \Ext_{k[K]}(k,k)$ by~\cite[Proposition~8.4.10]{BP}.
Because the homotopy fibration
\[ \mathcal{Z}_K \stackrel{\omega}{\longrightarrow} DJ_K \longrightarrow BT^m \]
admits a section after looping, the Pontryagin rings of the base, total space and fibre fit in an extension of Hopf algebras
\begin{equation} \label{eq_Hopf}
k \longrightarrow H_*(\Omega\mathcal{Z}_K;k) \longrightarrow H_*(\Omega DJ_K;k) \longrightarrow \Lambda(u_1,\ldots,u_m) \longrightarrow k.
\end{equation}
The Pontryagin ring of the moment-angle complex can therefore be viewed as a subalgebra 
\[
H_*(\Omega\mathcal{Z}_K;k)\subset H_*(\Omega DJ_K;k)\cong\Ext_{k[K]}(k,k).
\]
For each $i\in[m]$, let $\hat{\mu}_i \in \pi_1(\Omega DJ_K)$ be the adjoint of the canonical generator $\mu_i\in\pi_2(DJ_K)$ with Hurewicz image \[ u_i=h(\hat{\mu}_i) \in H_1(\Omega DJ_K;k) \cong \Ext^1_{k[K]}(k,k).\] 

Denote the homology suspension by $\sigma\colon H_{n-1}(\Omega X) \to H_n(X)$. Recall that $\sigma$ is induced by the evaluation map $\mathrm{ev}\colon \Sigma\Omega X \to X$, where $H_n(\Sigma\Omega X)\cong H_{n-1}(\Omega X)$. The image of $\sigma$ is contained in the submodule of primitives in $H_*(X)$ and contains the image of the Hurewicz map. To see the latter, note that any map $f\colon S^n \to X$ (with adjoint $\hat{f}\colon S^{n-1}\to \Omega X$) factors as a composite
\begin{equation} \label{eq_susp_hom}
\begin{tikzcd}
S^n \ar[r, "\Sigma\hat{f}"] \ar[rd,"f"'] & \Sigma\Omega X \ar[d,"\mathrm{ev}"] \\
 & X
\end{tikzcd}
\end{equation}
since $\mathrm{ev}$ is the counit of the adjunction $\Sigma \dashv \Omega$. So the Hurewicz image of any $f\in\pi_n(X)$ is contained in the image of the induced map $\mathrm{ev}_*=\sigma$, that is, $\sigma(h(\hat{f}))=h(f)$. 

The next corollary of \cref{thm_Hurewicz_equiv} identifies the sweep outs of any spherical homology class $h(f)\in H_n(\mathcal{Z}_K)$ with the images of certain commutators under the homology suspension map $\sigma\colon H_{*-1}(\Omega\mathcal{Z}_K) \to H_*(\mathcal{Z}_K)$.

\begin{corollary}\label{cor_equiv}
Let $f\in\pi_n(\mathcal{Z}_K)$. Then
\[ \sigma\big( [u_i,h(\hat{f})] \big) = \lambda_ih(f) \in H_{n+1}(\mathcal{Z}_K) \]
for all $i\in [m]$, where $[-,-]$ is the commutator in the Pontryagin ring $H_*(\Omega DJ_K)$.
\end{corollary}

\begin{proof}
The adjoint $\widehat{\omega\circ f} \in\pi_{n-1}(\Omega DJ_K)$ of $\omega\circ f \in \pi_n(DJ_K)$ is given by $\Omega\omega\circ\hat{f}$, and since $\Omega\omega$ induces the subalgebra inclusion $H_*(\Omega\mathcal{Z}_K) \to H_*(\Omega DJ_K)$ of~\eqref{eq_Hopf}, we denote the Hurewicz image of both $\hat{f}$ and $\widehat{\omega\circ f}$ simply by $h(\hat{f})$. The Samelson product
\[
\langle \hat{\mu}_i,\widehat{\omega\circ f} \rangle \colon S^1\wedge S^{n-1} \xrightarrow{\hat{\mu}_i\wedge\widehat{\omega\circ f}} \Omega DJ_K\wedge\Omega DJ_K \xrightarrow{[\;,\;]} \Omega DJ_K
\]
admits a lift $\ell\colon S^n \to \Omega\mathcal{Z}_K$ with Hurewicz image $h(\ell)$ given by the commutator $[u_i,h(\hat{f})] \in H_n(\Omega\mathcal{Z}_K)\subset H_n(\Omega DJ_K)$. Since $\langle \hat{\mu}_i,\widehat{\omega\circ f} \rangle$ is adjoint to the Whitehead product $[\mu_i,\omega\circ f] \in \pi_{n+1}(DJ_K)$, it follows that $\ell$ is adjoint to the unique lift $[\overline{\mu_i,\omega\circ f}] \in \pi_{n+1}(\mathcal{Z}_K)$ of the Whitehead product. Therefore, by the equation following~\eqref{eq_susp_hom}, $\sigma\big( [u_i,h(\hat{f})] \big)$ equals the Hurewicz image of $[\overline{\mu_i,\omega\circ f}]$, and the corollary follows by \cref{thm_Hurewicz_equiv}.
\end{proof}

\begin{remark}
In the Serre spectral sequence of the path space fibration 
\[ \Omega\mathcal{Z}_K \longrightarrow P\mathcal{Z}_K \longrightarrow \mathcal{Z}_K, \] 
every spherical class $h(f)\in H_n(\mathcal{Z}_K)$ transgresses to the Hurewicz image $h(\hat{f}) \in H_{n-1}(\Omega\mathcal{Z}_K)$ of the adjoint of $f$. It follows from \cref{cor_equiv} that each sweep out $\lambda_ih(f) \in H_{n+1}(\mathcal{Z}_K)$ transgresses to the commutator $[u_i,h(\hat{f})] \in H_n(\Omega\mathcal{Z}_K)$. 
\end{remark}

For each minimal non-face $I\notin K$, consider the natural map
\[
\omega_I\colon S^{2|I|-1} = \mathcal{Z}_{K_I} \to \mathcal{Z}_K
\]
induced by the inclusion $K_I \hookrightarrow K$. Since the Hurewicz images of these maps form a basis for $\Ext_S^1(k[K],k) \subseteq H_*(\mathcal{Z}_K;k)$, they also generate the quasi-linear strand of $k[K]$ as a $\Lambda$-module. 
Taking the maps $\omega_I$ as $f\in\pi_*(\mathcal{Z}_K)$ in \cref{thm_Hurewicz_equiv} and \cref{cor_equiv} therefore leads to a description of the lifts of iterated Whitehead products and iterated commutators
\[
[\overline{\mu_{j_1},[\mu_{j_2},\ldots[\mu_{j_r},\omega\circ\omega_I]\ldots]}] \in \pi_*(\mathcal{Z}_K), \qquad
[u_{j_1},[u_{j_2},\ldots[u_{j_r},h(\hat{w}_I)]\ldots]] \in H_*(\Omega\mathcal{Z}_K;k)
\]
that map onto the quasi-linear strand via the Hurewicz map $h\colon \pi_*(\mathcal{Z}_K) \to H_*(\mathcal{Z}_K;k)$ and homology suspension $\sigma\colon H_{*-1}(\Omega\mathcal{Z}_K;k) \to H_*(\mathcal{Z}_K;k)$, respectively.

Below we let $w_I = h(\hat{\omega}_I) \in H_{2|I|-2}(\Omega\mathcal{Z}_K;k)$
denote the Hurewicz image of the adjoint of $\omega_I$.

\begin{corollary} \label{cor_brackets}
Let $I\notin K$ be a minimal non-face of $K$ and let $j_1,\ldots,j_r \in [m]\smallsetminus I$. If the map
\[ 
\widetilde{H}_*(S^{|I|-2};k)=\widetilde{H}_*(K_I;k) \longrightarrow \widetilde{H}_*(K_{I\cup \{j_1,\ldots, j_r\}};k) 
\]
induced by the inclusion is nonzero, then the following hold: 
\begin{enumerate}[label={\normalfont(\arabic*)}]
\item the Whitehead product $[\mu_{j_1},[\mu_{j_2},\ldots[\mu_{j_r},\omega\circ\omega_I]\ldots]] \in \pi_*(DJ_K)$ is nonzero; \label{part1}
\item the commutator $[u_{j_1},[u_{j_2},\ldots[u_{j_r},w_I]\ldots]] \in H_*(\Omega \mathcal{Z}_K;k)\subset H_*(\Omega DJ_K;k)$ 
is nonzero. \label{part2}
\end{enumerate}
\end{corollary}

\begin{proof}
Under the isomorphism $H_*(\mathcal{Z}_K;k) \cong \bigoplus_{U\subseteq[m]}\widetilde{H}_*(K_U;k)$ of \cref{prop_homology_Z_K}, the iterated sweep out $\lambda_{j_1}\!\cdots\lambda_{j_r}h(\omega_I)$ of the generator $h(\omega_I) \in \Ext^1_S(k[K],k) \subseteq H_*(\mathcal{Z}_K;k)$ corresponds to the image of a generator under $\widetilde{H}_*(K_I;k) \to \widetilde{H}_*(K_{I\cup \{j_1,\ldots, j_r\}};k)$. Part~\ref{part1} therefore follows by iteratively applying \cref{thm_Hurewicz_equiv} and noting that the nontriviality of the Whitehead product is equivalent to the nontriviality of a lift in $\pi_*(\mathcal{Z}_K)$ since $\omega_*\colon \pi_*(\mathcal{Z}_K) \to \pi_*(DJ_K)$ is injective. Similarly, part~\ref{part2} follows by iteratively applying \cref{cor_equiv}.
\end{proof}

\begin{remark}
For a minimal non-face $I=\{i_1,\ldots,i_q\}\notin K$, Abramyan and Panov identify the natural map $\omega_I\colon S^{2|I|-1} = \mathcal{Z}_{K_I} \to \mathcal{Z}_K$ with the lift of a (non-iterated) higher Whitehead product $[\mu_{i_1},\ldots,\mu_{i_q}] \in \pi_*(DJ_K)$ in~\cite{AP}, where some combinatorial criteria are given for the nonvanishing of iterated higher Whitehead products of a particular form. We note that~\cite[Theorem~5.2]{AP} applies to the iterated Whitehead product in \cref{cor_brackets} only when $r=1$, implying that $[\mu_{j_1},\omega\circ\omega_I]=[\mu_{j_1},[\mu_{i_1},\ldots,\mu_{i_q}]]$ is trivial if and only if $K$ contains the cone $K_I*\{j_1\}$. Since the latter condition is equivalent to the inclusion $\partial\Delta^{|I|-1}=K_I \hookrightarrow K_{I\cup\{j_1\}}$ inducing the trivial map in reduced homology, \cref{cor_brackets} is a natural partial extension of the criterion in~\cite[Theorem~5.2]{AP}.
\end{remark}

\begin{remark}
When $I=\{i,j\} \notin K$ is a missing edge, $w_I\in H_2(\Omega\mathcal{Z}_K;k)$ is given by the commutator $[u_i,u_j]$, and by \cite[Theorem~4.3]{GPTW}, the iterated commutators $[u_{j_1},[u_{j_2},\ldots[u_{j_r},w_I]\ldots]]$ from \cref{cor_brackets}\ref{part2} generate $H_*(\Omega\mathcal{Z}_K;k)$ if $K$ is flag, or equivalently, $k[K]$ is Koszul. A similar statement holds for all Koszul quotient algebras $R=S/I$. Indeed, it is well known that $\Ext_R(k,k)\cong U(L)$ is the universal envelope of a graded Lie algebra $L$, and that when $R$ is Koszul $\Ext_R(k,k)$ is generated by the degree one elements $L^1$. It follows that $L^{\geqslant 2}$ is spanned by elements of the form
\[
[L^1,[L^1,\ldots[L^1,[L^1,L^1]]\ldots]]
\]
with bracket length at least two. The algebra $\Ext_{R\otimes^\mathsf{L}_Sk}(k,k)$ (which corresponds to $H_*(\Omega\mathcal{Z}_K;k)$ when $R=k[K]$ is a Stanley--Reisner ring) is isomorphic to the subalgebra $U(L^{\geqslant 2})$. It follows that $\Ext_{R\otimes^\mathsf{L}_Sk}(k,k)\cong U(L^{\geqslant 2})$ is generated by iterated brackets of elements of $L^1$.
\end{remark}

\section{Moment-angle manifolds associated to almost linear resolutions}
\label{sec_manifolds}

In this section we study the homotopy types of moment-angle manifolds associated to Stanley--Reisner ideals with almost linear resolutions. We begin by situating this class of ideals within a much larger one (which also includes componentwise almost linear ideals, see \cref{se_CAL_algebra}), to which our results also apply. 

Recall that $\mathcal{Z}_K$ is a manifold when $K$ is Gorenstein over $\mathbb{Z}$ (equivalently, over all fields $k$). In this case, $I_K$ is rarely\footnote{By a result of Nagel and R\"omer~\cite[Theorem~3.1]{NR}, a Gorenstein graded ideal $I\subset S$ is componentwise linear if and only if $I$ is a complete intersection whose minimal generators, except possibly one, are linear forms. Using this, it is easy to show that for a Gorenstein$^*$ comlpex $K$, $I_K$ is componentwise linear if and only if $\mathcal{Z}_K$ is homeomorphic to $S^{2n+1} \times\prod_{i=1}^rS^1$ for some $n,r\geqslant 0$. Note that linear forms appearing among the generators of $I_K$ requires $K$ to have ghost vertices.} componentwise linear, and $\mathcal{Z}_K$ cannot be homotopy equivalent to a wedge of spheres unless $K=\partial\Delta^n*\Delta^q$, in which case $\mathcal{Z}_K\simeq S^{2n+1}$. Nonetheless, the minimal free resolution of $I_K$ very often displays as much linearity as possible, in the sense of $I_K$ having maximal Green--Lazarsfeld index or maximal quasi-linear strand, subject to the Poincar\'e duality imposed by Gorensteinness. In these cases we show that the associated manifold $\mathcal{Z}_K$ has the homotopy type of a wedge of spheres with one cell attached.

A simplicial complex $K$ is called \emph{Gorenstein$^*$} over $k$ if it is Gorenstein over $k$, and $K$ is not a cone. Examples include all sphere triangulations and triangulations of homology spheres, more generally.

\subsection*{Almost linear resolutions and almost quasi-Koszul modules}

We continue to assume that $S=k[v_1,\ldots,v_m]$, although the following definition is sensible for any standard graded $k$-algebra~$S$.

\begin{definition}\label{def_AQK}
A finitely generated graded $S$-module $M$ is $\emph{almost quasi-Koszul}$ if the quasi-linear strand of $M$ contains $\Ext_S^{<p}(M,k)$, where $p=\pd_S(M)$ (see \cref{def_strand} for the definition of the quasi-linear strand).
\end{definition}

Recall from \cref{se_lin_background} that $M$ has an almost linear resolution when $M$ is generated in a single degree $d$ and $\Ext_S^{<p}(M,k)$ is concentrated along the linear strand $\bigoplus_{i\geqslant 0}\Ext_S^i(M,k)_{i+d}$. Since the linear strand agrees with the quasi-linear strand when $M$ is equigenerated, it follows that any $S$-module with an almost linear resolution is an almost quasi-Koszul module. The converse is not true (just as modules with linear resolutions are strictly contained in the class of quasi-Koszul modules, see \cref{se_lin_background}), and almost quasi-Koszul modules need not be equigenerated; see \cref{sec_examples} for many examples.

The following is analogous to the characterisation of quasi-Koszul Stanley--Reisner ideals in \cref{prop_HMF}. The \emph{deletion complex} $K\smallsetminus i$ is the full subcomplex of $K$ on the vertex set $[m]\smallsetminus\{i\}$.

\begin{proposition} \label{prop_almost_HMF}
Let $K$ be a Gorenstein$^*$ complex over $k$ on the vertex set $[m]$. Then the following are equivalent:
\begin{enumerate}[label={\normalfont(\alph*)}]
\item $I_K$ is an almost quasi-Koszul module; \label{aqk1}
\item $I_{K\smallsetminus i}$ is a quasi-Koszul module for all $i\in[m]$; \label{aqk2}
\item for each proper subset $U\subsetneq[m]$, $\widetilde{H}_*(K_U;k)$ is generated as a $k$-module by missing faces of $K$ (in the sense of~\cref{def_HMF}). \label{aqk3}
\end{enumerate}
\end{proposition}

\begin{proof}
Let $p=\pd_S(I_K)$. Since $K$ is Gorenstein$^*$, $\Ext_S^p(I_K,k)=\Ext_S^{p+1}(k[K],k)$ is a $1$-dimensional $k$-vector space corresponding to the summand $\widetilde{H}_*(K;k)$ in the Hochster formula, so we have
\[
\Ext_S^{<p+1}(k[K],k) \cong \bigoplus_{U\subsetneq[m]} \widetilde{H}_*(K_U;k).
\]
The equivalence of \ref{aqk1} and \ref{aqk3} therefore follows from \cref{prop_homology_Z_K} and the Hochster formula interpretation of the quasi-linear strand given in \cref{rem_Hoch_strand}. Since every full subcomplex $K_U$ with $U\subsetneq[m]$ is a full subcomplex of $K\smallsetminus i$ for some $i\in[m]$,  the equivalence of \ref{aqk2} and \ref{aqk3} follows from \cref{prop_HMF}. 
\end{proof}

Among Gorenstein complexes satisfying the equivalent conditions above, those for which $I_K$ has an almost linear resolution can be characterised in purely combinatorial terms. A simplicial $(n-1)$-sphere $K$ is \emph{neighbourly} if every set of $\lfloor\frac{n}{2}\rfloor$ vertices forms a face of $K$. We define neighbourly Gorenstein$^*$ complexes similarly.

\begin{proposition} \label{prop_alin}
Let $K\neq\partial\Delta^{m-1}$ be a Gorenstein$^*$ complex over $k$. Then $I_K$ has an almost linear resolution if and only if $K$ is odd-dimensional and neighbourly.
\end{proposition}

\begin{proof}
Let $K\neq\partial\Delta^{m-1}$ be an $(n-1)$-dimensional Gorenstein$^*$ complex on $m$ vertices. Since $k[K]$ is Cohen--Macaulay, we have $\pd_S(k[K])=m-n>1$.

Assume $I_K$ has an almost linear resolution, that is, $I_K$ satisfies condition $N_{d,m-n-1}$ for some $d\geqslant 1$. Then the nonzero Betti numbers of $k[K]$ are
\[
\beta_{0,0}(k[K]),\; \beta_{1,d}(k[K]),\; \beta_{2,d+1}(k[K]), \ldots, \beta_{m+n-1,d+m+n-2}(k[K]),\; \beta_{m-n,2d+m-n-2}(k[K]),
\]
where the bidegree of the top Betti number $\beta_{m-n,2d+m-n-2}(k[K]) = \beta_{0,0}(k[K])=1$ is determined by Poincar\'e duality. On the other hand, since the top Betti number corresponds to $\Tor^S_{m-n}(k[K],k)_m \cong \widetilde{H}^{n-1}(K;k)$ (or since the Castelnuovo--Mumford regularity $\operatorname{reg}_S(k[K])\coloneqq \sup\{j-i \,:\, \beta_{i,j}(k[K])\neq0\}$ is well-known to equal $n$), we must have that $2d+m-n-2=m$, so $n=2d-2$ and $K$ is odd-dimensional. Since every minimal non-face of $K$ has cardinality $d=\frac{n}{2}+1$, every set of $\frac{n}{2}$ vertices must be a face of $K$, so $K$ is neighbourly.

Conversely, assume $K$ is odd-dimensional and neighbourly. Since $K\neq\partial\Delta^{m-1}$, it follows that $I_K$ is generated in degree $d=\frac{n}{2}+1$. Therefore, by Poincar\'e duality, $\beta_{0,0}(k[K])$, $\beta_{1,d}(k[K])$, $\beta_{m-n-1,m-d}(k[K])$ and $\beta_{m-n,m}(k[K])$ are the only nonzero Betti numbers in homological degrees $0$, $1$, $m-n-1$ or $m-n$. To show that $I_K$ has an almost linear resolution, it remains to show that the intermediary Betti numbers are concentrated along the linear strand $\{\beta_{i,i+d-1}(k[K])\}$. Suppose toward contradiction that $\beta_{i,i+j}(k[K])\neq 0$ for some $1<i<m-n-1$ and $j>d-1$. Then by Poincar\'e duality, $\beta_{m-n-i,m-i-j}\neq 0$. But $m-i-j-(m-n-i)=n-j<n-(d-1)=d-1$ (since $n=2d-2$), so this is a contradiction as $\beta_{i,i+j}(k[K])=0$ for $j<d-1$ and $i>0$.
\end{proof}

\begin{remark}
It follows from \cref{prop_alin} that, among sphere triangulations $K$, the property of $I_K$ having an almost linear resolution is independent of the base field $k$. For arbitrary simplicial complexes $K$, this no longer holds: if $K$ is the $6$-vertex triangulation of $\mathbb{R}P^2$, then $I_K$ has a linear (and therefore almost linear) resolution when $\operatorname{char}k\neq 2$, but it does not have an almost linear resolution when $\operatorname{char}k=2$.
\end{remark}

\subsection*{Homotopy types of moment-angle manifolds and their loop spaces}

For the remainder of this section, $I_K$ should be understood as a module over the ring $S=\mathbb{Z}[v_1,\ldots,v_m]$ of polynomials over the integers. We simply call $K$ a \emph{Gorenstein$^*$ complex} if it is Gorenstein$^*$ over $\mathbb{Z}$ (equivalently, over all fields). In this case, it is easy to see that $I_K$ is an almost quasi-Koszul $S$-module if and only if $I_K$ is an almost quasi-Koszul $k[v_1,\ldots,v_m]$-module for all fields $k$.

\begin{theorem} \label{thm_conn_sum}
Let $K\neq\partial\Delta^{m-1}$ be an $(n-1)$-dimensional Gorenstein$^*$ complex on~$m$ vertices. If $I_K$ is an almost quasi-Koszul module, then the following hold:
\begin{enumerate}[label={\normalfont(\arabic*)}]
\item $\mathcal{Z}_K$ is an $(m+n)$-manifold with $(m+n-1)$-skeleton homotopy equivalent to a wedge of spheres; \label{item1}
\item $\mathcal{Z}_K$ is rationally homotopy equivalent to a connected sum of sphere products with two spheres in each product; \label{item2}
\item $\Omega\mathcal{Z}_K$ is homotopy equivalent to a product of spheres and loop spaces of spheres.
\end{enumerate}
\end{theorem}

\begin{proof}
The assumptions on $K$ imply that the moment-angle manifold  is $(m+n)$-dimensional and $\pd_S(\mathbb{Z}[K])=m-n$, so we have that $H_*(\operatorname{sk}_{m+n-1}\mathcal{Z}_K) \cong \Ext_S^{<m-n}(\mathbb{Z}[K],\mathbb{Z})$. Since $I_K$ is almost quasi-Koszul, this portion of $H_*(\mathcal{Z}_K)\cong \Ext_S(\mathbb{Z}[K],\mathbb{Z})$ is contained in the quasi-linear strand and hence in the image of the Hurewicz map by \cref{thm_Hurewicz_strand}. Therefore $\operatorname{sk}_{m+n-1}\mathcal{Z}_K$ is homotopy equivalent to a wedge of spheres.

Since $\mathcal{Z}_K$ is a $2$-connected $(m+n)$-dimensional Poincar\'e duality complex with $(m+n-1)$-skeleton having the homotopy type of a wedge of spheres, it follows immediately from~\cite[Proposition~3.2]{BT} that $\Omega\mathcal{Z}_K$ is homotopy equivalent to a product of spheres and loop spaces of spheres.

It remains to prove~\ref{item2}. By Hochster's formula and the Gorenstein condition, $H^*(\mathcal{Z}_K;\mathbb{Q})\cong \Tor^S(\mathbb{Q}[K],\mathbb{Q})$ is a multigraded Poincar\'e duality algebra concentrated in squarefree multidegrees. In particular, any homogeneous element $x$ of a given (nonzero) multidegree satisfies $x^2=0$. Moreover, since $\operatorname{sk}_{m+n-1}\mathcal{Z}_K$ has the homotopy type of a wedge of spheres, it is possible to choose a symplectic basis for $H^*(\mathcal{Z}_K;\mathbb{Q})$, that is, a basis $\{1,x_1,\ldots,x_g,y_1,\ldots,y_g,z\}$ where $z$ is a generator of $H^{m+n}(\mathcal{Z}_K;\mathbb{Q})\cong \Tor^S_{m-n}(\mathbb{Q}[K],\mathbb{Q})_{[m]}$ and the only nonzero products among the positive degree basis elements are $x_iy_i=z$ for $i=1,\ldots,g$. Therefore $\mathcal{Z}_K$ has the cohomology ring of a connected sum of sphere products $\#_{i=1}^g (S^{|x_i|}\times S^{|y_i|})$, and both of these spaces are homotopy equivalent to $(\bigvee_{i=1}^gS^{|x_i|}\vee S^{|y_i|}) \cup_f e^{m+n}$ for some attaching map $f$. A result of Stasheff~\cite[Theorem~1]{St} on the uniqueness of the rational homotopy type of a simply connected $N$-dimensional Poincar\'e duality complex with a given cohomology ring and $(N-1)$-skeleton now implies the desired equivalence $\mathcal{Z}_K \simeq_\mathbb{Q} \#_{i=1}^g (S^{|x_i|}\times S^{|y_i|})$.
\end{proof}

An immediate application of \cref{thm_conn_sum} is the determination of the formality and Lusternik--Schnirelmann category for a large class of moment-angle manifolds.

\begin{corollary}
Let $K\neq\partial\Delta^{n-1}$ be a Gorenstein$^*$ complex. If $I_K$ is an almost quasi-Koszul module, then $\mathcal{Z}_K$ is a formal manifold with $\operatorname{cat}(\mathcal{Z}_K)=2$. \qed
\end{corollary}

Minimally non-Golod complexes, introduced by Berglund and J\"ollenbeck in~\cite{BJ}, are simplicial complexes $K$ for which $\mathbb{Z}[K]$ is not Golod but $\mathbb{Z}[K\smallsetminus i]$ is Golod for all $i\in[m]$. Since the boundary complexes of stacked polytopes have almost quasi-Koszul Stanley--Reisner ideals (as we will see in \cref{sec_examples}), the next corollary is a vast generalisation of~\cite[Theorem~5]{BJ}.

\begin{corollary} \label{cor_mnGolod}
Let $K\neq\partial\Delta^{n-1}$ be a Gorenstein$^*$ complex. If $I_K$ is an almost quasi-Koszul module, then $K$ is minimally non-Golod.
\end{corollary}

\begin{proof}
Since $\operatorname{sk}_{m+n-1}\mathcal{Z}_K$ is homotopy equivalent to a wedge of spheres by \cref{thm_conn_sum}, this follows exactly as in the proof of~\cite[Theorem~1.1]{Am}.
\end{proof}

Panov and Theriault~\cite[Theorem~1.1]{PT} show that if $K$ is flag, then the natural map $\mathcal{Z}_{K^{(0)}} \to \mathcal{Z}_K$ admits a section (that is, a right homotopy inverse) after looping, where $K^{(0)}$ is the $0$-skeleton of $K$ consisting of $m$ disjoint vertices, implying that $\Omega\mathcal{Z}_K$ is a homotopy retract of $\Omega\mathcal{Z}_{K^{(0)}}$. Gorenstein complexes for which $I_K$ is almost quasi-Koszul provide a large class of simplicial complexes for which an analogous statement holds. Let $K^{(d)}$ denote the $d$-skeleton of $K$ consisting of all faces of dimension at most $d$ in $K$.

\begin{corollary} \label{cor_section}
Let $K\neq\partial\Delta^{n-1}$ be a Gorenstein$^*$ complex. If $I_K$ is an almost quasi-Koszul module with highest degree generator in degree $d+2$, then the natural map $\mathcal{Z}_{K^{(d)}} \to \mathcal{Z}_K$ admits a section after looping.
\end{corollary}

\begin{proof}
Assume $K$ is $(n-1)$-dimensional and has $m$ vertices, so that $\mathcal{Z}_K$ is an $(m+n)$-manifold. Consider the diagram
\[
\begin{tikzcd}
\mathcal{Z}_{K^{(d)}} \ar[r] & \mathcal{Z}_K\\
& \operatorname{sk}_{m+n-1}(\mathcal{Z}_K), \ar[ul,dashed] \ar[u]
\end{tikzcd}
\]
where the right vertical map induces the inclusion of the quasi-linear strand of $\mathbb{Z}[K]$ in homology, and $\operatorname{sk}_{m+n-1}(\mathcal{Z}_K)$ is a wedge of spheres up to homotopy by \cref{thm_conn_sum}\ref{item1}. To show that a lift exists making the diagram homotopy commute, we will show that each sphere in this wedge lifts through the top horizontal map. It follows from the proof of \cref{thm_Hurewicz_strand} (see also the discussion preceding \cref{cor_brackets}) that the homotopy equivalence in \cref{thm_conn_sum}\ref{item1} can be chosen so that each wedge summand maps to $\mathcal{Z}_K$ via a lifted Whitehead product $[\overline{\mu_{j_1},[\mu_{j_2},\ldots[\mu_{j_r},\omega\circ\omega_I]\ldots]}]$, where $I$ is a minimal non-face of $K$. Since $|I|\leqslant d+2$, every proper subset of $I$ is a face of $K^{(d)}$, so $I$ is also a minimal non-face of $K^{(d)}$. This implies that $K_I=K^{(d)}_I$, the inclusion $K_I\hookrightarrow K$ factors through $K^{(d)}$, and hence $\omega_I\colon S^{2|I|-1}=\mathcal{Z}_{K_I} \to \mathcal{Z}_K$ factors through $\mathcal{Z}_{K^{(d)}}$. Since each $\mu_j$ factors through $DJ_{K^{(d)}}$, the map $[\mu_{j_1},[\mu_{j_2},\ldots[\mu_{j_r},\omega\circ\omega_I]\ldots]]$ and its lift factor through $DJ_{K^{(d)}}$ and $\mathcal{Z}_{K^{(d)}}$, respectively, by naturality of the Whitehead product.

Finally, the right vertical map in the diagram admits a section after looping by~\cite[Proposition~3.2(ii)]{BT}, so the same is true of the top horizontal map by commutativity of the diagram.
\end{proof}

A graded $k$-algebra $A$ is called a \emph{$\mathcal{K}_2$ algebra} if the Yoneda algebra $\Ext_A(k,k)$ is generated in cohomological degrees $1$ and $2$. In addition to the Stanley--Reisner rings of flag complexes (which are Koszul algebras and therefore $\mathcal{K}_2$), it was shown in~\cite[Theorem~1.3]{CS} that $k[K]$ is a $\mathcal{K}_2$ algebra when $K^\vee$ is sequentially Cohen--Macaulay over $k$.

\begin{corollary}
Let $K$ be a Gorenstein$^*$ complex. If $I_K$ is an almost quasi-Koszul module, then $H_*(\Omega\mathcal{Z}_K;k)$ is a one-relator algebra multiplicatively generated by the iterated commutators $[u_{j_1},[u_{j_2},\ldots[u_{j_r},w_I]\ldots]]$ of \cref{cor_brackets}, and $k[K]$ is a $\mathcal{K}_2$ algebra.
\end{corollary}

\begin{proof}
As in the proof of \cref{cor_section}, the homotopy equivalence $\bigvee_iS^{n_i} \simeq \operatorname{sk}_{m+n-1}(\mathcal{Z}_K)$ of \cref{thm_conn_sum}\ref{item1} can be chosen so that each sphere maps to $\mathcal{Z}_K$ by a lifted Whitehead product $[\overline{\mu_{j_1},[\mu_{j_2},\ldots[\mu_{j_r},\omega\circ\omega_I]\ldots]}]$, and $\operatorname{sk}_{m+n-1}(\mathcal{Z}_K) \to \mathcal{Z}_K$ admits a section after looping. By the Bott--Samelson theorem, $H_*\bigl(\Omega\bigvee_iS^{n_i};k\bigr)$ is generated by the Hurewicz images of the obvious adjoint maps $S^{n_i-1} \to \Omega\bigvee_iS^{n_i}$. Since the loop map $\Omega\operatorname{sk}_{m+n-1}(\mathcal{Z}_K) \to \Omega\mathcal{Z}_K$ induces a surjective ring homomorphism in homology, it follows that $H_*(\Omega\mathcal{Z}_K;k)$ is generated by the Hurewicz images of the adjoints of the lifted Whitehead products, which are the commutators $[u_{j_1},[u_{j_2},\ldots[u_{j_r},w_I]\ldots]]$. An Adams--Hilton model computation exactly as in~\cite[Proposition~4.1]{GISP} shows that $H_*(\Omega\mathcal{Z}_K;k)$ is a one-relator algebra. Note that by~\eqref{eq_Hopf} the algebra $H_*(\Omega DJ_K;k) \cong \Ext_{k[K]}(k,k)$ is generated by generators of $H_*(\Omega\mathcal{Z}_K;k)$ together with the classes $u_i\in \Ext^1_{k[K]}(k,k)$. Since $w_I\in H_*(\Omega\mathcal{Z}_K;k) \subset \Ext_{k[K]}(k,k)$ is of cohomological degree $1$ (resp. $2$) for a minimal non-face of size $|I|=2$ (resp. $|I|>2$), we conclude that $k[K]$ is $\mathcal{K}_2$.
\end{proof}

\section{Componentwise almost linear ideals}
\label{se_CAL_algebra}

In this section we turn our attention to the more general algebraic setting of modules over Koszul rings. We will later specialise these results to Stanley--Reisner ideals in polynomial rings, in particular obtaining a characterisation of the Stanley--Reisner ideals with componentwise almost linear resolutions.

In \cref{se_lin_background} we stated a result of R\"omer characterising componentwise linear modules in terms of their Koszulity; see \cref{thm_complin_iff_Kos}. However, this result is of limited use in the case of Stanley--Reisner ideals since (for instance) Gorenstein ideals in polynomial ideals can never have componentwise linear resolutions unless they are generated by a single element.

We recall from \cref{se_lin_background} that a finitely generated graded $S$-module $M$ is called {componentwise linear in the first $r$ steps} if, for each $d$, the submodule $M_{\langle d\rangle}$ generated by degree $d$ elements of $M$ has a resolution that is linear in the first $r$ steps. 
The main goal of this section is to prove an analogue of R\"omer's Theorem for componentwise almost linear modules.

\begin{theorem} \label{thm_complin_steps}
Assume that $S$ is Koszul and let $M$ be a finitely generated graded $S$-module with a minimal free resolution $F\to M$. Fix $r>0$. Then $M$ is componentwise linear in the first $r+1$ steps if and only if  $H_n(\lin(F))=0$ for $0<n< r$ and (when $r\neq 0$) $H_0(\lin(F))= \gr(M)$.
\end{theorem}

\begin{remark}
The last condition of the theorem says, more precisely, that the map $H_0(\lin(F))\to \gr(M)$ of (\ref{eq_H0lin}) from \cref{se_lin_background} is an isomorphism. Equivalently,  $H_0(F_{\langle d\rangle})=M_{\langle \leqslant d\rangle}/M_{\langle \leqslant d-1\rangle}$ for all $d$. In the work \cite{IR}, this condition that $H_0(\lin(F))$ is ``as expected'' is not needed, since it is implied by exactness of $\lin(F)$. As we only assume exactness of $\lin(F)$ up to a certain degree, this condition is necessary in our setting (see \cref{ex_AQK}).
\end{remark}

Before proving the theorem we need a few lemmas. We also remind the reader that the  $d$-linear strand $F_{\langle d\rangle}$ was defined for minimal complexes of free graded $S$-modules  in \cref{se_lin_background}.

\begin{lemma}\label{lem_lowest_strand}
    Let $S$ be a standard graded Koszul algebra and $M$ be a finitely generated graded $S$-module concentrated in degrees $d$ and above. If $F \to M$ and $G\to M_{\langle d\rangle}$ are minimal resolutions, then the induced map $G_{\langle d\rangle} \to  F_{\langle d\rangle}$ is an isomorphism.
\end{lemma}

\begin{proof} 
By Nakayama's lemma it suffices to check that $\Tor^S_n(M,k)_{n+d} \to \Tor^S_n(M_{\langle d\rangle},k)_{n+d}$ is an isomorphism for each $n$. As $S$ is Koszul the minimal resolution $L\to k$ is linear, and so $\Tor^S_n(M,k)_{n+d} = H_n(M\otimes_SL)_{n+d} = H_n(M_{\langle d\rangle}\otimes_SL)_{n+d} = \Tor^S_n(M_{\langle d\rangle},k)_{n+d}$.
\end{proof}

The next lemma is analogous to \cite[Lemma 5.5]{IR}:

\begin{lemma}\label{lem_linear_factor}
Let $S$ be a standard graded Koszul algebra and $M$ be a finitely generated graded $S$-module concentrated in degrees $d$ and above. The following are equivalent:
\begin{enumerate}[label={\normalfont(\alph*)}]
\item\label{linlem1} $M$ is componentwise linear in the first $r$ steps;
\item\label{linlem2} $M_{\langle d\rangle}$ is linear in the first $r$ steps and $M/M_{\langle d\rangle}$ is componentwise linear in the first $r$ steps.
\end{enumerate}
\end{lemma}

\begin{proof}
Under either \ref{linlem1} or \ref{linlem2} $M_{\langle d\rangle}$ is linear in the first $r$ steps by assumption, and we start by showing that this implies $(M_{\langle d\rangle})_{\langle i\rangle}$ is linear in the first $r$ steps, for all $i$. If $d \leq i\leq j$ then $((M_{\langle d\rangle})_{\langle i\rangle})_{\langle j\rangle} = (M_{\langle d\rangle})_{\langle j\rangle}$, so it suffices by induction to treat the case $i=d+1$.

We consider the exact sequence of Tor groups associated to the following short exact sequence
\[
0\longrightarrow(M_{\langle d\rangle})_{\langle d+1\rangle} \longrightarrow M_{\langle d\rangle} \longrightarrow M_{\langle d\rangle}/(M_{\langle d\rangle})_{\langle d+1\rangle} \longrightarrow0.
\]
 Since $(M_{\langle d\rangle})_{\langle d+1\rangle}$ is generated in degree $d+1$, the right-hand map below must be zero for $n<r$
\[
\Tor^S_{n+1}(M_{\langle d\rangle}/(M_{\langle d\rangle})_{\langle d+1\rangle},k)\longrightarrow\Tor^S_{n}((M_{\langle d\rangle})_{\langle d+1\rangle} ,k)\xrightarrow{\ 0\ } \Tor^S_{n}(M_{\langle d\rangle} ,k).
\]
Note that $M_{\langle d\rangle}/(M_{\langle d\rangle})_{\langle d+1\rangle}$ is a direct sum of copies of $k$ shifted into degree $d$, and therefore, since $R$ is Koszul, it has a linear resolution. Hence the first map above being surjective shows that $(M_{\langle d\rangle})_{\langle d+1\rangle}$ is linear in the first $r$ steps.

We finish the proof using this exact sequence (cf.\ \cite[Lemma 5.5]{IR}):
\[
0\longrightarrow(M_{\langle d\rangle})_{\langle i\rangle}\longrightarrow M_{\langle i\rangle} \longrightarrow (M/M_{\langle d\rangle})_{\langle i\rangle}\longrightarrow0.
\]
If \ref{linlem1} holds then, as we have just seen, the first two terms of this sequence are linear in the first $r$ steps. Therefore, in the resulting exact sequence
\[
\Tor^S_{n}(M_{\langle i\rangle},k)\longrightarrow\Tor^S_{n}((M/M_{\langle d\rangle})_{\langle i\rangle} ,k)\xrightarrow{\ 0\ } \Tor^S_{n-1}((M_{\langle d\rangle})_{\langle i\rangle},k)
\]
the second map is zero for $n-1<r$. As before, surjectivity of the first map in this range makes $(M/M_{\langle d\rangle})_{\langle i\rangle}$ linear in the first $r$ steps.

By a similar token, if \ref{linlem2} holds then the outer two terms of the sequence are linear in the first $r$ steps. The exact sequence that results from applying $\Tor^S(-,k)$ shows immediately that $M_{\langle i\rangle}$ is linear in the first $r$ steps.
\end{proof}

\begin{proof}[Proof of \cref{thm_complin_steps}]

Let $d=\min\{i~:~ M_i\neq 0\}$ and write
\[
{\rm reg}_{\leq r}(M) = \max\left\{ e-d ~:~ (F_n)_{\langle e+n\rangle} \neq 0\text{ for some } 0\leq n\leq r\right\}
\]
We'll do induction on this invariant: if ${\rm reg}_{\leq r}(M) =0$ then $F$ is $d$-linear in the first $r+1$ steps and $M/M_{\langle d\rangle}=0$, so $M$ is componentwise linear in the first $r+1$ steps by \cref{lem_linear_factor}.
Likewise, in this situation $\lin(F)_{\leq r}\cong F_{\leq r}$, and therefore $H_n(\lin(F))=0$ for $0<n<r$ and $H_0(\lin(F))=M$. So the two conditions in the theorem are vacuously equivalent.

Now assume that ${\rm reg}_{\leq r}(M)$ is strictly positive. In this case we consider the ``lowest linear strand'' $F_{\langle d \rangle}$ of $F$ defined by
\[
(F_{\langle d \rangle})_n\coloneqq (F_n)_{\langle d+n\rangle}
\]
for $n\geq 0$. Note that $F_{\langle d \rangle}$ is finitely generated and free in each homological degree, and, for degree reasons, $F_{\langle d \rangle}$ is a subcomplex of $F$. Note also that $H_0(F_{\langle d \rangle})=M_{\langle d \rangle}$. By definition there is a short exact sequences of complexes of free $R$-modules
\begin{equation}\label{eq_lin_ses}
0\to F_{\langle d \rangle} \to  F\to F/ F_{\langle d \rangle}\to 0
\end{equation}
that canonically splits in each homological degree. This yields a canonical decomposition of complexes
\begin{equation}\label{eq_lin_decomp}
    \lin(F) = \lin(F/ F_{\langle d \rangle}) \oplus \lin(F_{\langle d \rangle}) = \lin(F/ F_{\langle d \rangle}) \oplus F_{\langle d \rangle}
\end{equation}
Assume that $M$ is componentwise linear in the first $r+1$ steps. In particular, $M_{\langle d\rangle}$ has a resolution that is  $d$-linear in the first $r+1$ steps. By \cref{lem_lowest_strand} this implies $H_n(F_{\langle d \rangle})=0$ for $0<n<r$.

The exact sequence in homology associated to \eqref{eq_lin_ses} shows that $H_n(F/F_{\langle d\rangle})=0$ for $0<n\leq r$, and $H_0(F/F_{\langle d\rangle})=M/M_{\langle d\rangle}$.
Therefore, if $G\to M/M_{\langle d\rangle}$ is a minimal free resolution, then $G_{\leq r}\cong (F/F_{\langle d\rangle})_{\leq r}$. From this we see that ${\rm reg}_{\leq r}(M/M_{\langle d\rangle})<{\rm reg}_{\leq r}(M)$, which will allow us to use induction. We also see that $\lin(G)_{\leq r}\cong 
\lin(F/F_{\langle d\rangle})_{\leq r}$.  Combining this with our induction hypothesis yields $H_i(\lin(F/F_{\langle d\rangle}))\cong H_i(\lin(G)) =0$ for $0<i<r$ and $H_0(\lin(F/F_{\langle d\rangle}))=\gr(M/M_{\langle d\rangle})$. Finally, using \eqref{eq_lin_decomp}, it follows that $H_n(\lin(F))=0$ for $0<n<r$  and $H_0(\lin(F))= \gr(M/M_{\langle d\rangle})\oplus \gr(M_{\langle d\rangle})=\gr(M)$, establishing the forward direction of the theorem.

Assume conversely that $H_n(\lin(F))=0$ for $0<n<r$  and (provided $r\neq 0$) $H_0(\lin(F))=\gr(M)$.  From \eqref{eq_lin_decomp} it follows that $H_n(F_{\langle d\rangle})=0$ for $0<n<r$. Since $H_0(F_{\langle d\rangle})=M_{\langle d\rangle}$, this means $M_{\langle d\rangle}$ is $d$-linear in the first $r+1$ steps.

Just as before, we use the homology exact sequence associated to \eqref{eq_lin_ses} to see
that $H_n(F/F_{\langle d\rangle})=0$ for $0<n\leq r$ and $H_0(F/F_{\langle d\rangle})=M/M_{\langle d\rangle}$, so we deduce that ${\rm reg}_{\leq r}(M/M_{\langle d\rangle})<{\rm reg}_{\leq r}(M)$. From \eqref{eq_lin_decomp} we have $H_n(\lin(F/F_{\langle d\rangle}))=0$ for $0<n<r$ and $H_0(\lin(F/F_{\langle d\rangle}))=M/M_{\langle d\rangle}$. We also note that for $i>d$ we have $H_0(F_{\langle i\rangle}) = M_{\langle  \leqslant i\rangle}/M_{\langle  \leqslant i-1\rangle}= (M/M_{\langle d\rangle})_{\langle  \leqslant i\rangle}/(M/M_{\langle d\rangle})_{\langle  \leqslant i-1\rangle}$, which means $H_0(\lin(F/F_{\langle d\rangle}))=\gr(M/M_{\langle d\rangle})$. The three previous sentences taken together mean we can apply our induction hypothesis  to see that $M/M_{\langle d\rangle}$ is componentwise linear in the first $r+1$ steps. 
 
Combining the previous two paragraphs,  $M$ is componentwise linear in the first $r+1$ steps by \cref{lem_linear_factor}. This completes that backward direction of the theorem, and closes the induction.
\end{proof}

The next result is an algebraic generalisation of \cref{lem_Kos_qKos}. It is similar to \cite[Proposition~3.4]{CS}.

\begin{proposition}\label{prop_qlinstrand}
    Let $S$ be a standard graded ring, and let $M$ be a finitely generated graded $S$-module. If $M$ is componentwise linear in the first $p+1$ steps, then the quasi-linear strand of $M$ contains $\Ext_S^{\leqslant p}(M,k)$.
\end{proposition}

\begin{proof}
We first assume that $M$ is generated in a single degree $d$. Let $F$ be the minimal fee resolution of $M$, so that $\Ext^i_S(M,k)={\rm Hom}_S(F_i,k)$. Assume that $F_i$ is generated in degree $d+i$ for $i< p$. We will show that this implies $\Ext_S^{i+1}(M,k)=\Ext_S^1(k,k)\cdot \Ext_S^i(M,k)$.

Since $F$ is a minimal resolution, for each minimal generator $f\in F_{i+1}$ we have $\partial(f)\neq 0$. For degree reasons, we may write $\partial(f)=\sum x_j g_j$ with $x_j\in S_1$ and $\{g_j\}$ forming a basis of $F_{i}$. Pick $j$ so that $x_j\neq 0$, and let $\theta: F_i\to k(-i-d)$ be dual to $g_j$ (so $\theta(g_j)=1$ and $\theta(g_k)=0$ when $j\neq k$). Also let $\sigma\in \Ext_S^1(k,k)\cong (S_1)^*$ be represented by a map $\sigma \colon S_1\to k$ be such that $\sigma( x_j)=1$. Calculating the Yoneda action shows that the element $\sigma\cdot\theta\in \Ext_S^{i+1}(M,k)={\rm Hom}_S(F_{i+1},k)$ satisfies  $(\sigma\cdot\theta)(f)=1$. Since $f$ was arbitrary, this shows that the elements of $\Ext_S^1(k,k)\cdot \Ext_S^i(M,k)$ do not simultaneously vanish on any minimal generator of $F_{i+1}$, and therefore they span $\Ext_S^{i+1}(M,k)={\rm Hom}_S(F_{i+1},k)$.

In case $M$ is not generated in a single degree, we proceed by induction on the number of degrees in which $M$ is generated, the base having just been established. If $d$ is the lowest degree of a minimal generator of $M$, then the exact sequence
\[
0\to M_{\langle d\rangle}\to M \to M/ M_{\langle d\rangle}\to 0
\]
Induces an exact sequence
\[
\Ext_S^i(M_{\langle d\rangle},k)\leftarrow \Ext_S^i(M,k) \leftarrow \Ext_S^i(M/ M_{\langle d\rangle},k).
\]
Note that when $i=0$ the left map is surjective. 
It follows that if, for $i<p$, the outer two terms can be generated from classes of degree zero under the action of $\Ext^1_S(k,k)$, then so can the term in them middle. Since both $M_{\langle d\rangle}$ and $M/M_{\langle d\rangle}$ are generated in fewer degrees than $M$, the induction hypothesis finishes the proof.
\end{proof}

We draw attention to an important special case of componentwise linearity:

\begin{definition}
Let $S$ be a standard graded Koszul $k$-algebra, and let $M$ be an $S$-module of finite projective dimension. We say that $M$ is \emph{componentwise almost linear} if it is componentwise linear in the first $r=\pd_S(M)$ steps.

If $R$ is a standard graded ring, then there is an essentially unique way to write $R=S/I$ as a quotient of a standard graded polynomial ring $S$ having $\dim(R_1)=\dim(S_1)$. We will say that $R$ is \emph{componentwise almost linear} (CAL) if the ideal $I$ is componentwise almost linear as an $S$-module.
\end{definition}

The next result is a direct consequence of \cref{thm_complin_steps} and \cref{prop_qlinstrand}. We remind the reader that almost quasi-Koszul modules were defined in \cref{def_AQK}; in this context, it means that $\Ext_S(R,k)$ is generated in degrees $0$ and $1$ as a module over the exterior algebra $\Lambda=\Ext_S(k,k)$, except for possibly the top degree.

\begin{corollary} \label{cor_cal_aqk}
        Let $R$ be a standard graded ring, and write $R=S/I$, where $S$ is a standard graded polynomial ring having the same embedding dimension of $R$, and assume $p=\pd_S(R)>2$. Then $R$ is componentwise almost linear if and only if the minimal $S$-free resolution $F\to R$ satisfies $H_1(\lin(F))= \gr(I)$ and  $H_i(\lin(F))=0$ for $1<i<p-1$. Moreover, these conditions imply that $I$ is an almost quasi-Koszul $S$-module.
\end{corollary}

\section{Examples of almost quasi-Koszul \texorpdfstring{$I_K$}{IK}} \label{sec_examples}

The results above show that there are strong algebraic and topological consequences of almost linearity and its generalisations, especially in the Gorenstein setting. Our aim in this section is to show that almost quasi-Koszulity is extremely common, by producing many examples of Gorenstein complexes that have almost linear, componentwise almost linear, or almost quasi-Koszul Stanley--Reisner ideals.

\subsection*{Connected sums of simplicial complexes} We start by showing that almost quasi-Koszulity is closed under taking connected sums.

Let $L$ and $L'$ be two $(n-1)$-dimensional simplicial complexes, and let $\sigma$ be a common facet (i.e., an $(n-1)$-dimensional face of both $L$ and $L'$, with a chosen bijection between their vertices). The \emph{connected sum} $L\#_\sigma L'$ of $L$ and $L'$ along $\sigma$ is the simplicial complex obtained from the disjoint union of $L$ and $L'$ by deleting $\sigma$ from both and identifying the common vertices of $\sigma$.

\begin{remark}
    There is a direct algebraic analogue of the connected sum construction. For this we fix a pair of surjective homomorphisms $R\to T\leftarrow R'$ of commutative rings, a $T$-module $V$, and homomorphisms $R\xleftarrow{\iota_R} V\xrightarrow{\iota_{R'}} R'$ of $R$- and $R'$-modules respectively.  From this data the authors of \cite{AAM} define the \emph{connected sum} of $R$ and $R'$ over $T$ to be
    \[
    R\#_TR'=\left(R\times_TR'\right)/\big\{ (\iota_{R}(v),\iota_{R'}(v))\mid v\in V\big\}.
    \]
    Given $L$, $L'$ and $\sigma$ as above, this construction applied to Stanley--Reisner rings recovers $L\#_\sigma L$. We take $R=k[L]$ and $R'=k[L']$, and labeling the vertices of $\sigma$ by $\{1,\ldots,n\}$, we take $T=k[v_1,\ldots,v_n]$. The homomorphisms $R\to T\leftarrow R'$ are the identity on the variables $v_1,\ldots,v_n$, and zero on all other variables. We also take $V=T$ and define $\iota_R(1)=v_1\cdots v_n\in R$ and $\iota_{R'}(1)=v_1\cdots v_n\in R'$. With these in place, one may check directly that
    \[
    k[L]\#_Tk[L']\cong k[L\#_\sigma L'].
    \]
\end{remark}

To simplify terminology, below we will call a simplicial complex $K$ \emph{almost quasi-Koszul} over $k$ if $I_K$ is an almost quasi-Koszul $S$-module, where $S=k[v_1,\ldots,v_m]$ and $k$ is a field or the ring of integers $\mathbb{Z}$. (In light of \cref{prop_almost_HMF}, one could also call these \emph{almost HMF complexes}.) Denote the set of minimal non-faces of $K$ by $\mathrm{MF}(K)$.

\begin{theorem} \label{thm_closed_under_sums}
For any $n>1$, the collection of $(n-1)$-dimensional Gorenstein$^*$ complexes which are almost quasi-Koszul over $k$ is closed under connected sums. 
\end{theorem}

\begin{proof}
Let $\sigma$ be a common facet of Gorenstein$^*$ complexes $L$ and $L'$ with $\dim\sigma=n-1$. Assume that $L$ and $L'$ are almost quasi-Koszul over $k$ and let $K=L\#_\sigma L'$. Then by~\cite[Corollary~3.10]{MM}, $K$ is a Gorenstein$^*$ complex of dimension $n-1$. By relabeling vertices, we may assume $K$ has vertex set~$[m]$. Choose subsets $U,V\subseteq [m]$ so that
\[ U\cup V=[m], \quad U\cap V=\sigma, \quad K_U=L\smallsetminus\sigma, \quad K_V=L'\smallsetminus\sigma. \]
Note that $K_U\cap K_V=K_{U\cap V}=\partial\Delta^{n-1}$. 

Our goal is to show for each nonzero $\alpha \in H^*(\mathcal{Z}_K;k)\cong \Tor^S(k[K],k)$ of homological degree $1<i<m-n$ that $\iota_i\alpha \neq 0$ for some $i\in[m]$. Equivalently, by \cref{prop_homology_Z_K}, it suffices to show for each nonempty subset $I\subsetneq [m]$ with $I\notin \mathrm{MF}(K)$ that every nonzero $\alpha \in \widetilde{H}^*(K_I;k)$ maps nontrivially under the map $\widetilde{H}^*(K_I;k) \to \widetilde{H}^*(K_{I\smallsetminus i};k)$ induced by the inclusion $K_{I\smallsetminus i} \hookrightarrow K_I$ for some $i\in I$. 

Let $\alpha \in \widetilde{H}^*(K_I)$ be such a class. (All cohomology groups are taken with coefficients in $k$, which we henceforth suppress from the notation.) Consider the Mayer--Vietoris sequence
\[ 
\cdots\longrightarrow \widetilde{H}^{*-1}(K_{I\cap U\cap V}) \stackrel{\delta}{\longrightarrow} \widetilde{H}^*(K_I) \longrightarrow \widetilde{H}^*(K_{I\cap U})\oplus\widetilde{H}^*(K_{I\cap V}) \longrightarrow\cdots
\]
associated to the cover $\{K_{I\cap U}, K_{I\cap V}\}$ of the full subcomplex $K_I$. Since $K_{I\cap U\cap V}$ is a full subcomplex of $K_{U\cap V}=\partial\Delta^{n-1}$, it is contractible unless $U\cap V\subseteq I$, in which case $|K_{I\cap U\cap V}|=|K_{U\cap V}|=S^{n-2}$. So if $\alpha$ is in the image of the connecting map $\delta$, then $\alpha\in \widetilde{H}^{n-1}(K_I) \cong \Tor^S_{|I|-n}(k[K],k)_I$ lies in the $n$-linear strand $\bigoplus_i\Tor^S_{i-n}(k[K],k)_i$. But, by Poincar\'e duality, the only nonzero summand in this strand is the socle $\Tor^S_{m-n}(k[K],k)_{[m]}$. It follows that $I=[m]$ and $\alpha$ is of homological degree~$m-n$, contradicting our assumption. Therefore~$\alpha$ is not in the image of $\delta$ and, by exactness, maps nontrivially under the map induced by $K_{I\cap U} \hookrightarrow K_I$ or $K_{I\cap V} \hookrightarrow K_I$. Without loss of generality, assume the former. If $I\not\subseteq U$, then the inclusion $K_{I\cap U} \hookrightarrow K_I$ factors through $K_{I\smallsetminus i} \hookrightarrow K_I$ for some $i\in I\smallsetminus U$, and $\alpha$ maps nontrivially under the map induced by $K_{I\smallsetminus i} \hookrightarrow K_I$, as desired. Next, assuming $I\subseteq U$, we consider two cases.

\emph{Case 1}: $\sigma \not\subseteq I$. In this case, we have $K_I=L_I$. Therefore $\alpha$ maps nontrivially under the map induced by $K_{I\smallsetminus i} \hookrightarrow K_I$ for some $i\in I$, since $L$ is almost quasi-Koszul.

\emph{Case 2}: $\sigma \subseteq I$. Observe that, in this case, $K_I$ is not a full subcomplex of $L$ since $\sigma$ is a missing face of $K_I$ and a facet of $L$. However, since $L_I$ is obtained from $K_I$ by adding the missing face $\sigma$, the inclusion $\partial\Delta^{n-1}=K_\sigma \hookrightarrow K_I$ gives a cofibration
\[ S^{n-2} = |\partial\Delta^{n-1}| \longrightarrow |K_I| \longrightarrow |L_I|. \]
If $\alpha$ maps nontrivially under the map induced by $K_\sigma \hookrightarrow K_I$, then either $I=\sigma$, contradicting the assumption that $I\notin \mathrm{MF}(K)$, or else $\sigma$ is a proper subset of $I$ and $\alpha$ maps nontrivially under the map induced by $K_{I\smallsetminus i} \hookrightarrow K_I$ for some $i\in I\smallsetminus \sigma$, as desired. On the other hand, if $\alpha$ is in the kernel of the map induced by $K_\sigma \hookrightarrow K_I$, then $\alpha$ is the image of some $\beta\in \widetilde{H}^*(L_I)$ in the exact sequence induced by the cofibration above. Note that if $I\in \mathrm{MF}(L)$, then $\sigma \subsetneq I$ and consequently $L_I=\partial\Delta^n$ (since $|I|>n$ and $\dim L=n-1$). But this implies that $K_I=L_I\smallsetminus\sigma$ is contractible, contradicting the assumption that $\alpha\neq 0$. Therefore $I\notin \mathrm{MF}(L)$, and since $L$ is almost quasi-Koszul, it follows that $\beta$ maps to some nonzero $\gamma\in \widetilde{H}^*(L_{I\smallsetminus i})$ under the map $\widetilde{H}^*(L_I) \to \widetilde{H}^*(L_{I\smallsetminus i})$ induced by $L_{I\smallsetminus i} \hookrightarrow L_I$ for some $i\in I$. Finally, $\gamma$ maps nontrivially under the map induced by $K_{I\smallsetminus i} \hookrightarrow L_{I\smallsetminus i}$ as this induces an injection in $\widetilde{H}^{<n-1}(-)$. To see this, note that $K_{I\smallsetminus i}=L_{I\smallsetminus i}$ if $i\in \sigma$, and otherwise $L_{I\smallsetminus i}$ is obtained from $K_{I\smallsetminus i}$ by adding the missing face $\sigma$, leading to a cofibration as above. By commutativity of the square
\[
\begin{tikzcd}
K_I \ar[r,hook] & L_I \\
K_{I\smallsetminus i} \ar[u,hook] \ar[r,hook] & L_{I\smallsetminus i} \ar[u,hook]
\end{tikzcd}
\]
we conclude that $\alpha\in \widetilde{H}^*(K_I)$ has nontrivial image under the map induced by the left vertical inclusion, which finishes the proof.
\end{proof}

Since sphere triangulations are Gorenstein$^*$ and closed under connected sums, the next statement is an immediate consequence of \cref{thm_closed_under_sums}.

\begin{corollary}\label{cor_sphere_sums}
The collection of simplicial $(n-1)$-spheres which are almost quasi-Koszul over $k$ is closed under connected sums. \qed
\end{corollary}

We draw attention to another special case of \cref{thm_closed_under_sums}, this time by observing that stellar subdivision of an $(n-1)$-dimensional face $\sigma$ of $K$ coincides with taking the connected sum $K\#_\sigma \partial\Delta^n$ with the boundary of an $n$-simplex.

\begin{corollary}\label{cor_stellar}
The collection of Gorenstein$^*$ complexes which are almost quasi-Koszul over $k$ is closed under stellar subdivision of facets, and in particular includes the boundary complexes of all stacked polytopes.
\end{corollary}

\subsection*{The class of almost quasi-Koszul complexes}
We finish by surveying some examples of simplicial complexes to which \cref{thm_conn_sum} and its corollaries apply. Let $\mathcal{C}$ denote the collection of Gorenstein$^*$ complexes that are almost quasi-Koszul over $\mathbb{Z}$.

For the examples in this section we will use a shorthand denoting monomials in $S=k[v_1,\ldots,v_m]$ by concatenating the corresponding indices, so that for example
\[
I=(123,124,35,46,56)\quad\text{denotes the ideal}\quad \left(v_{1}v_{2}v_{3},\,v_{1}v_{2}v_{4},\,v_{3}v_{5},\,v_{4}v_{6},\,v_{5}v_{6}\right).
\]

\begin{example}[Cyclic polytopes] \label{ex_cyclic}
The cycle graph $K=C_m$, regarded as a $1$-dimensional simplicial complex given by the boundary of an $m$-gon, is the typical example of a sphere triangulation for which $I_K$ has an almost linear resolution. Following \cite{BH}, we say a pure graded free resolution $(F,\partial)$ has \emph{degree type} $(d_p,\ldots,d_1)$ if $d_i$ is the degree of each nonzero entry of a matrix representing the map $\partial_i\colon F_i\to F_{i-1}$ for all $1\leqslant i\leqslant p$, where $p$ is the length of the resolution. Since $C_m$ is flag, the minimal free resolution of its Stanley--Reisner ring has degree type $(\underbrace{2,1,\ldots,1,2}_{m-2})$. 

More generally, let $K=\partial C_{2n}(m)$ be the boundary complex of the cyclic polytope $C_{2n}(m)$, defined as the convex hull of $m$ distinct points along the moment curve $t\mapsto (t,t^2,\ldots,t^{2n}) \in \mathbb{R}^{2n}$. Then $K$ is a neighbourly $(2n-1)$-sphere triangulation, so $I_K$ has an almost linear resolution by \cref{prop_alin}. Degree types $(n+1,1,\ldots,1,n+1)$ of any length and any $n\geqslant 1$ are realised by the Stanley--Reisner rings of this class of examples. 

The boundary complexes $K=\partial C_{2n+1}(m)$ of odd-dimensional cyclic polytopes are also neighbourly, but since they have minimal non-faces of size $n+1$ and $n+2$ (when $m>2n+2$ so $K\neq\partial\Delta^{2n+1}$), their Stanley--Reisner ideals are not equigenerated and hence do not have almost linear resolutions. The minimal free resolutions of these $I_K$ are more complicated and were worked out in~\cite{BoP}, where it can be checked using the explicit formulas for the differentials in~\cite[Section~2]{BoP} that $K=\partial C_{2n+1}(m)$ is almost quasi-Koszul over all fields, and hence over $\mathbb{Z}$. Therefore the boundary complexes of all cyclic polytopes belong to $\mathcal{C}$.
\end{example}

\begin{example}[Stacked polytopes] \label{ex_stacked}
Another family of simplicial polytopes naturally generalising polygons are given by stacked polytopes. Their boundary complexes are constructed by starting with $\partial\Delta^n$ and stellar subdividing facets in any order. Alternatively, these are the boundary complexes $K=\partial P^*$ of the duals of simple polytopes $P=\operatorname{vc}^k(\Delta^n)$ obtained by applying any sequence of vertex cuts to a simplex. The boundary complexes of all stacked polytopes belong to $\mathcal{C}$ by \cref{cor_stellar}.
\end{example}

\begin{example}[Generalised truncation polytopes] 
Both families of examples above include simplicial complexes of the form $K=\partial\Delta^n * \partial\Delta^{n+1}$ (equivalently, $K=\partial P^*$ where $P=\Delta^n\times\Delta^{n+1}$). More generally, $I_K$ is clearly componentwise almost linear for any $K=\partial\Delta^{n_1}*\partial\Delta^{n_2}$, so applying stellar subdivision to the facets of such $K$ leads to another infinite family of complexes belonging to $\mathcal{C}$ by \cref{cor_stellar}. Since all complexes in $\mathcal{C}$ are minimally non-Golod by \cref{cor_mnGolod}, this recovers a result of Limonchenko~\cite[Theorem~3.2]{L} on \emph{generalised truncation polytopes} $P=\operatorname{vc}^k(\Delta^{n_1}\times\cdots\times\Delta^{n_r})$ which shows that $K_P\coloneqq \partial P^*$ is minimally non-Golod when $r\in \{1,2\}$ and $P$ is not a simplex. 

As observed in~\cite{L}, the moment-angle manifolds $\mathcal{Z}_{K_P}$ associated to the generalised truncation polytopes $P=\operatorname{vc}^k(\Delta^{n_1}\times\Delta^{n_2})$ are diffeomorphic to connected sums of sphere products for $0\leqslant k< 2(n_1+n_2-1)$ by~\cite[Theorem~2.2]{GL}. We point out that this bound on the number of vertex cuts $k$ can be removed using a conjecture of Gitler and L\'opez de Medrano that was proved in~\cite{CFW}.
\end{example}

\begin{example}[Neighbourly sphere triangulations] \label{ex_nbrly}
The examples above, and indeed all sphere triangulations $K$ for which the homotopy type of the manifold $\mathcal{Z}_K$ has been determined to date, are boundary complexes of certain simplicial polytopes. It is well known, however, that most sphere triangulations are not polytopal. Let $s(n,m)$ and $c(n,m)$ be the number of combinatorial types of simplicial $(n-1)$-spheres and simplicial $n$-polytopes, respectively, on the vertex set $[m]$. Then by a result of Kalai~\cite{K}, $\lim_{m\to\infty}\frac{c(n,m)}{s(n,m)} = 0$ for every $n\geqslant 5$.

On the other hand, the number of neighbourly sphere triangulations grows very quickly with~$n$ and~$m$: for $n\geqslant 5$, the best-known asymptotic lower bound on the number of neighbourly simplicial $(n-1)$-spheres on the vertex set $[m]$ is $2^{\Omega(m^{\lfloor (n-1)/2\rfloor})}$ \cite{NZ}. (By comparison, for $n\geqslant 4$ Stanley's upper bound theorem implies $s(n,m)\leqslant 2^{O(m^{\lfloor n/2\rfloor}\log m)}$~\cite[Section~4.2]{K}, and $c(n,m)=2^{\Theta(m\log m)}$ by~\cite{Al,GP}). In fact, for any $n$, it is conjectured that most triangulations of the $(2n-1)$-sphere with sufficiently many vertices are neighbourly (see~\cite[Section~6.3]{K}). Since simplicial spheres $K$ which are odd-dimensional and neighbourly are precisely those for which $I_K$ has an almost linear resolution by \cref{prop_alin}, they belong to $\mathcal{C}$ and the topological results of \cref{sec_manifolds} can be viewed as describing the homotopy types of (conjecturally) generic moment-angle manifolds.
\end{example}

\begin{example}[The Br\"uckner--Gr\"unbaum sphere]
By Steinitz's theorem, every simplicial $2$-sphere can be realised as the boundary complex of a $3$-polytope. The first example of a non-polytopal simplicial sphere is an $8$-vertex triangulation of the $3$-sphere known as the \emph{Br\"uckner--Gr\"unbaum sphere} \cite{GSr}, which can be defined as the link of any vertex in the minimal $9$-vertex triangulation of $\mathbb{C}P^2$. Its minimal non-faces can be read off from a list of its $20$ facets, from which we obtain the Stanley--Reisner ideal 
\[
I_K= (125,135,136,138,145,146,246,247,256,257,278,348,357,378,468,678).
\]
Since $I_K$ has no generators of degree two, $K$ is neighbourly and hence $K\in \mathcal{C}$ by \cref{prop_alin}. We can also see directly that $K$ is almost linear, since, using {\tt Macaulay2}, the  Betti table of its Stanley--Reisner ring is given by:
         \[\beta_{i,j}\big(S/I_K\big):\qquad 
         \begin{array}{c|c c c c c}
         { \scriptstyle j-i}^{\ i}\!& 0 & 1 & 2 & 3 & 4\\ \hline

         0 & 1 & . & . & . & .\\
         1 & . & . & . & . & .\\
         2 & . & 16 & 30 & 16 & .\\
         3 & . & . & . & . & .\\
         4 & . & . & . & . & 1
\end{array}
\hspace{25mm}
\]
\end{example}

The next two examples illustrate that each of the inclusions
\[
\left\{ \substack{I_K\text{ with almost}\\ \text{linear resolution}} \right\} \subset
\left\{ \substack{\text{componentwise}\\ \text{almost linear }I_K} \right\} \subset
\left\{ \substack{\text{almost quasi-}\\ \text{Koszul }I_K} \right\}
\]
are proper (as is the case with the word ``almost" removed, cf. \cref{se_lin_background}).

\begin{example}
The complete intersection ideal $I_K$ associated to any $K=\partial\Delta^{n_1}*\partial\Delta^{n_2}$ with $n_1\neq n_2$ is clearly componentwise almost linear, but does not have an almost linear resolution since $I_K$ is not equigenerated. A less trivial example is given by the Stanley--Reisner ideal $I_K=(123,124,36,45,56)$
of the stacked polytope $K$ obtained from $\partial\Delta^3$ by stellar subdividing two of its facets (i.e., $K=\partial P^*$ where $P=\operatorname{vc}^2(\Delta^3)$). 
\end{example}

\begin{example}
\label{ex_AQK}
To begin with, let $K'=\partial \Delta^2\ast\partial \Delta^2$ be the simplicial complex with Stanley--Reisner ideal $I_{K'} = (123,456)$ in $S'=k[v_1,\ldots,v_6]$. Then $I_{K'}$ has an almost linear resolution (and in particular, is almost quasi-Koszul). Now let $K$ be the stellar subdivision of $K'$ at the facet $1245$. The corresponding Stanley--Reisner ideal in $S=k[v_1,\ldots,v_7]$ is
\[
I_K=(1245,123,456,37,67),
\]
and it follows from \cref{cor_stellar} that $K$ is almost quasi-Koszul. We can also use {\tt Macaulay2} to see this algebraically: the minimal free resolution $F$ of $k[K]$ is
\[
{S^{1}}
\xleftarrow{\left(\!\begin{smallmatrix}v_{3}v_{7}&v_{6}v_{7}&v_{1}v_{2}v_{3}&v_{4}v_{5}v_{6}&v_{1}v_{2}v_{4}v_{5}\end{smallmatrix}\!\right)}
        {S^{5}}
        \xleftarrow{\left(\!\begin{smallmatrix}
        -v_{6}&-v_{1}v_{2}&0&0&0\\
        v_{3}&0&-v_{4}v_{5}&0&0\\
        0&v_{7}&0&-v_{4}v_{5}&0\\
        0&0&v_{7}&0&-v_{1}v_{2}\\
        0&0&0&v_{3}&v_{6}
        \end{smallmatrix}\!\right)}\,{S^{5}}
        \xleftarrow{\left(\!\begin{smallmatrix}
        -v_{1}v_{2}v_{4}v_{5}\\
        v_{4}v_{5}v_{6}\\
        -v_{1}v_{2}v_{3}\\
        v_{6}v_{7}\\
        -v_{3}v_{7}
        \end{smallmatrix}\!\right)}{S^{1}}\xleftarrow{}{0}.
        \]
        The action of $\Lambda=\Ext_S(k,k)$ on $\Tor^S(k[K],k)$ can be computed by following linear terms in the differential of $F$, see \cref{rem_linear_part_action}, and hence the linear entries in each row of the $5\times 5$ matrix show directly that $k[K]$ is almost quasi-Koszul.

        Despite this, $I_K$ is not componentwise almost linear; if it were then for all $n$ the ideal $(I_K)_{\langle n\rangle}$, generated by elements in $I_K$ of degree $n$, would have a resolution that is linear in the first $2$ steps. However, again using {\tt Macaulay2}, the ideal $(I_K)_{\langle 3\rangle}=$
          \[
          \left(\begin{smallmatrix}
        v_{1}v_{2}v_{3},\,v_{4}v_{5}v_{6},\,v_{1}v_{3}v_{7},\,v_{2}v_{3}v_{7},\,v_{3}^{2}v_{7},\,v_{3}v_{4}v_{7},\,v_{3}v_{5}v_{7},\,v_{1}v_{6}v_{7},\,v_{2}v_{6}v_{7},\,v_{3}v_{6}v_{7},\,v_{4}v_{6}v_{7},\,x_{5}v_{6}v_{7},\,v_{6}^{2}v_{7},\,v_{3}v_{7}^{2},\,v_{6}v_{7}^{2}\end{smallmatrix}\right)
          \]
        has a free resolution with Betti table
\[             \beta_{i,j}\big(S/(I_K)_{\langle 3\rangle}\big):\qquad \begin{array}{c|c c c c c c c c}
    { \scriptstyle j-i}^{\ i}\! & 0 & 1 & 2 & 3 & 4 & 5 & 6 & 7\\ \hline 
        0& 1 & . & . & . & . & . & . & .\\
        1 & . & . & . & . & . & . & . & .\\
        2 & . & 15 & 44 & 70 & 70 & 42 & 14 & 2\\
        3 & . & . & . & . & . & . & . & .\\
        4 & . & . & 1 & 1 & . & . & . & .
        \end{array}
        \hspace{25mm}
        \]
The fact that $\beta_{2,6}(S/(I_K)_{\langle 3\rangle})\neq 0$ shows that $I_K$ is not componentwise almost linear.

This example also shows that \cref{thm_closed_under_sums} and its corollaries do not hold for the almost linear or componentwise almost linear properties.
\end{example}

\begin{example}[Triangulations of $S^3$]
According to \cite{LUTZ} there are 1343 triangulations of $S^3$ with 9 or fewer vertices, listed on the \emph{Manifold Page} of Frank Lutz at \cite{ManifoldPage}. By \cref{thm_conn_sum} we can identify the rational homotopy type of the corresponding moment-angle manifold whenever they are almost quasi-Koszul. Fortunately, it is feasible to check all of these simplical complexes using {\tt Macaulay2}, and it turns out that 144 of the 1343 triangulations are almost quasi-Koszul (only 57 of these are neighbourly), so that the class $\mathcal{C}$ is already quite large when restricted to this number of vertices. As mentioned in \cref{ex_nbrly}, it is expected that most odd-dimensional sphere triangulations are almost quasi-Koszul when the number of vertices is sufficiently large.
\end{example}

\begin{example}[Triangulations of homology spheres]
The class $\mathcal{C}$ also contains many Gorenstein$^*$ complexes $K$ which are not sphere triangulations. To see this, note that any triangulation of a homology sphere is Gorenstein$^*$, so for example any neighbourly triangulation of the Poincar\'e homology sphere is in $\mathcal{C}$ by \cref{prop_alin}. We note that every homology sphere of dimension three admits a neighbourly triangulation by~\cite[Corollary~5.15]{Sw}; in particular, for any perfect $3$-manifold group $G$ there exist complexes $K\in\mathcal{C}$ with $\pi_1(|K|)\cong G$. Further families of such $K$ which are not neighbourly but still in $\mathcal{C}$ can then be constructed using \cref{cor_stellar}, since $\pi_1(|K|)$ is invariant under stellar subdivision.
\end{example}

\end{document}